\documentclass{amsart}
\usepackage{amsmath,amsfonts,amssymb,amsthm,floatrow,caption,subcaption, algorithm, algorithmic, graphics,enumerate,booktabs}
\usepackage[colorlinks,citecolor=blue,urlcolor=blue]{hyperref}
\usepackage{tikz}
\usetikzlibrary{positioning}
\usepackage[round]{natbib}
\input xy
\xyoption{all}
\newtheorem{thm}{Theorem}[section]
\newtheorem{lem}[thm]{Lemma}
\newtheorem{prop}[thm]{Proposition}
\newtheorem{cor}[thm]{Corollary}

\theoremstyle{definition}
\newtheorem{defn}[thm]{Definition}

\newtheorem{exmp}[thm]{Example}

\theoremstyle{remark}
\newtheorem{rem}[thm]{Remark}

\input{commands.tex}
\newcommand{\bi}{\leftrightarrow}
\newcommand{\Left}{\text{Left}}
\newcommand{\Right}{\text{Right}}

\allowdisplaybreaks

\numberwithin{equation}{section}

\title[Determinantal generalizations of instrumental variables]{Determinantal generalizations of instrumental variables}

\author[Weihs, Robinson, Dufresne et al.]{Luca Weihs}
\address{Department of Statistics, University of Washington, Seattle,
WA, U.S.A.}  \email{lucaw@uw.edu}

\author[]{Bill Robinson} \address{Department of Mathematics,
University of Maryland University College, Marlboro, MD, U.S.A.}
\email{robins.wm@gmail.com}

\author[]{Emilie Dufresne} \address{School of Mathematical Sciences,
University of Nottingham, Nottingham, UK}
\email{emilie.dufresne@nottingham.ac.uk}

\author[]{Jennifer Kenkel} \address{Department of Mathematics,
University of Utah, Salt Lake City, UT, U.S.A.}
\email{kenkel@math.utah.edu}

\author[]{Kaie Kubjas} \address{Department of Mathematics and Systems
Analysis, Aalto University, Espoo, Finland}
\email{kaie.kubjas@aalto.fi}

\author[]{Reginald L. McGee II} \address{Mathematical Biosciences
Institute, Columbus, Ohio, U.S.A.}  \email{mcgee.278@mbi.osu.edu}

\author[]{Nhan Nguyen} \address{Department of Mathematics, University
of Montana, Missoula, Montana, U.S.A.}
\email{nhan.nguyen@umontana.edu}

\author[]{Elina Robeva} \address{Department of Mathematics,
Massachusetts Institute of Technology, Cambridge, MA, U.S.A.}
\email{erobeva@mit.edu}

\author[]{Mathias Drton} \address{Department of Statistics, University
of Washington, Seattle, WA, U.S.A.}
\email{md5@uw.edu}

\date{\today}                                

\begin{document}

\begin{abstract}
  Linear structural equation models relate the components of a random
  vector using linear interdependencies and Gaussian noise.  Each such
  model can be naturally associated with a mixed graph whose vertices
  correspond to the components of the random vector.  The graph
  contains directed edges that represent the linear relationships
  between components, and bidirected edges that encode unobserved
  confounding.  We study the problem of generic identifiability, that
  is, whether a generic choice of linear and confounding effects can
  be uniquely recovered from the joint covariance matrix of the
  observed random vector.  An existing combinatorial criterion for
  establishing generic identifiability is the half-trek criterion
  (HTC), which uses the existence of trek systems in the mixed graph
  to iteratively discover generically invertible linear equation
  systems in polynomial time. By focusing on edges one at a
  time, we establish new sufficient and necessary conditions for
  generic identifiability of edge effects extending those of the
  HTC.  In particular, we show how edge coefficients can be recovered
  as quotients of subdeterminants of the covariance matrix,
  which constitutes a determinantal generalization of formulas
  obtained when using instrumental variables for identification.
\end{abstract}


\maketitle

\section{Introduction}

In a \emph{linear structural equation model} (L-SEM) the joint
distribution of a random vector $X=(X_1,\dots,X_n)^T$ obeys noisy
linear interdependencies. These interdependencies can be expressed
with a matrix equation of the form
\begin{align}
  X = \lambda_0 + \Lambda^TX + \epsilon, \label{eq:sem_eq}
\end{align}
where $\Lambda=(\lambda_{vw})\in\mathbb{R}^{n\times n}$ and
$\lambda_0 = (\lambda_{01},\dots,\lambda_{0n})^T\in\bR^n$ are unknown
parameters, and $\epsilon = (\epsilon_1,\dots,\epsilon_n)^T$ is a
random vector of error terms with positive definite covariance matrix
$\Omega = (\omega_{vw})$.  Then $X$ has mean vector
$(I-\Lambda)^{-T}\lambda_0$ and covariance matrix
\begin{align} \label{eq:sigma}
  \phi(\Lambda,\Omega) := (I-\Lambda)^{-T}\Omega(I-\Lambda)^{-1} = \Sigma
\end{align}
where $I$ is the $n\times n$ identity matrix. L-SEMs have been widely
applied in a variety of settings due to the clear causal
interpretation of their parameters
\citep{bollen,spirtes2000,pearl2009}.

Following an approach that dates back to \cite{wright1921,
  wright1934}, we may view $\Lambda$ and $\Omega$ as (weighted)
adjacency matrices corresponding to directed and bidirected graphs,
respectively.  This yields a natural correspondence between L-SEMs and
\emph{mixed graphs}, that is, graphs with both directed edges,
$v\to w$, and bidirected edges, $v\bi w$. More precisely, the mixed
graph $G$ is associated to the L-SEM in which $\lambda_{vw}$ is
assumed to be zero if $v\to w\not\in G$ and, similarly,
$\omega_{vw}=0$ when $v\bi w\not\in G$.  We write $\phi_G$ for the map
obtained by restricting the map $\phi$ from~(\ref{eq:sigma}) to pairs
$(\Lambda,\Omega)$ that satisfy the conditions encoded by the graph
$G$.  We note that mixed graphs used to represent L-SEMs are often
also called \emph{path diagrams}.

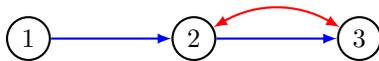
\begin{figure}[t]
  \centering
  \tikzset{
    every node/.style={circle, inner sep=1mm, minimum size=0.55cm, draw, thick, black, fill=white, text=black},
    every path/.style={thick}
  }
  \begin{tikzpicture}[align=center,node distance=2.2cm]
    \node [] (1) [] {1};
    \node [] (2) [right of=1]    {2};
    \node [] (3) [right of=2]    {3};

    \draw[blue] [-latex] (1) edge (2);
    \draw[blue] [-latex] (2) edge (3);
    \draw[red] [latex-latex, bend left] (2) edge (3);
  \end{tikzpicture}

  \caption{The mixed graph for the instrumental variable model.} \label{fig:instrumental-variables}
\end{figure}

\begin{exmp}
  The mixed graph in Figure \ref{fig:instrumental-variables} corresponds to the well-known instrumental variable model \citep{didelez2010}. In equations, this model asserts that
  \begin{align*}
    X_1 = \lambda_{01} + \epsilon_1, \quad X_2 = \lambda_{02} + \lambda_{12}X_1 + \epsilon_2, \text{ and} \quad X_3 = \lambda_{03} + \lambda_{23}X_2 + \epsilon_3,
  \end{align*}
  where $\epsilon$ has 0 mean and covariance matrix
  \begin{align*}
    \Omega = \left(
    \begin{matrix}
      \omega_{11} & 0 & 0 \\
      0 & \omega_{22} & \omega_{23} \\
      0 & \omega_{23} & \omega_{33}
    \end{matrix}\right).
  \end{align*}
  In this model, the random vector $X=(X_1,X_2,X_3)$ has covariance matrix
  \begin{align*}
    \Sigma &= \left(
    \begin{matrix}
      1 & -\lambda_{12} & 0 \\
      0 & 1 & -\lambda_{23} \\
      0 & 0 & 1
    \end{matrix}\right)^{-T}
              \left(
              \begin{matrix}
                \omega_{11} & 0 & 0 \\
                0 & \omega_{22} & \omega_{23} \\
                0 & \omega_{23} & \omega_{33}
              \end{matrix}\right)
                                  \left(
                                  \begin{matrix}
                                    1 & -\lambda_{12} & 0 \\
                                    0 & 1 & -\lambda_{23} \\
                                    0 & 0 & 1
                                  \end{matrix}\right)^{-1}\\
    &= \left(
\begin{array}{ccc}
 \omega_{11} & \lambda_{12} \omega_{11} & \lambda_{12} \lambda_{23} \omega_{11} \\
 \lambda_{12} \omega_{11} & \omega_{11} \lambda_{12}^2+\omega_{22} & \lambda_{23}
   \omega_{11} \lambda_{12}^2+\lambda_{23} \omega_{22}+\omega_{23} \\
 \lambda_{12} \lambda_{23} \omega_{11} & \lambda_{23} \omega_{11} \lambda
  _{12}^2+\lambda_{23} \omega_{22}+\omega_{23} & w_{33}+2\omega_{23}\lambda_{23}+\lambda_{23}^2\sigma_{22} \\
\end{array}
\right).
  \end{align*}
\end{exmp}
\medskip

A first question that arises when specifying an L-SEM via a mixed
graph $G$ is whether the map $\phi_G$ is injective, that is, whether any
$(\Lambda,\Omega)$ in the domain of $\phi_G$ can be uniquely recovered
from the covariance matrix $\phi_G(\Lambda,\Omega)$.  When this
injectivity holds we say that the model and also simply the graph $G$
is \emph{globally identifiable}.  Whether or not 
global identifiability holds can be decided in polynomial time
\citep{drton2011,shpitser:2006,tian:pearl:2002}.  However, in
many cases global identifiability is too strong a condition.  Indeed,
the canonical instrumental variables model is not globally
identifiable.

We will be instead interested in \emph{generic identifiability}, that
is, whether $(\Lambda,\Omega)$ can be recovered from
$\phi_G(\Lambda,\Omega)$ with probability 1 when choosing
$(\Lambda,\Omega)$ from any continuous distribution on the domain of
$\phi_G$.  A current state-of-the-art, polynomial time verifiable,
criterion for checking generic identifiability of a given mixed graph
is the half-trek criterion (HTC) of \citet{halftrek}, with
generalizations by \cite{chen2014,chen:2015,drton2016}.  The
sufficient condition that is part of the HTC operates by iteratively
discovering invertible linear equation systems in the $\Lambda$
parameters which it uses to prove generic identifiability.  A
necessary condition given by the HTC detects cases in which the
Jacobian matrix of $\phi_G$ fails to attain full column rank which
implies that the parameterization $\phi_G$ is generically
infinite-to-one.  However, there remain a considerable number of cases
in which the HTC remains inconclusive, that is, the graph satisfies
the necessary but not the sufficient condition for generic
identifiability.

We extend the applicability of the HTC in two ways.  First, we show
how the theorems on trek separation by \cite{sullivant2010} can be
used to discover determinantal relations that in turn can be used to
prove the generic identifiability of individual edge coefficients in
L-SEMs.  This method generalizes the use of conditional independence
in known instrumental variable techniques; compare
e.g.~\cite{Brito02}.  Once we have shown that individual edges are
generically identifiable with this new method, it would be ideal if
identified edges could be integrated into the equation systems
discovered by the HTC to prove that even more edges are generically
identifiable.  Unfortunately, the HTC is not well suited to integrate
single edge identifications as it operates simultaneously on all edges
incoming to a given node.  Our second contribution resolves this issue
by providing an \emph{edgewise} half-trek criterion which operates on
subsets of a node's parents, rather than all parents at once.  This
edgewise criterion often identifies many more coefficients than the
usual HTC.  We note that, in the process of preparing this manuscript
we discovered independent work of \cite{chen2016}; some of our results
can be seen as a generalization of results in his work.

The rest of this paper is organized as follows. In Section
\ref{sec:pre}, we give a brief overview of the necessary background on
mixed graphs, L-SEMs, and the half-trek criterion.  In Section
\ref{sec:trek-sep}, we show how trek-separation allows the generic
identification of edge coefficients as quotients of
subdeterminants. We introduce the edgewise half-trek criterion in
Section \ref{sec:edgewise-id} and we discuss necessary conditions for
the generic identifiability of edge coefficients in Section
\ref{sec:edgewise-non-iden}. Computational experiments showing the
applicability of our sufficient conditions follow in Section
\ref{sec:comp}, and we finish with a brief conclusion in Section
\ref{sec:conclusion}.  Some longer proofs are deferred to the
appendix.

\section{Preliminaries} \label{sec:pre}

We assume some familiarity with the graphical representation of
structural equation models and only give a brief overview of our
objects of study.  A more in-depth introduction can be found, for example, in
\cite{pearl2009} or, with a focus on the linear case considered here,
in \cite{drton:review:2017}.

\subsection{Mixed Graphs and Covariance Matrices} \label{sec:mixed-graphs}

Nonzero covariances in an L-SEM may arise through direct or
through confounding effects.  Mixed graphs with two types of edges
have been used to represent these two sources of dependences.

\begin{defn}[Mixed Graph]
  A \emph{mixed graph} on $n$ vertices is a triple $G=(V,D,B)$ where
  $V=\{1,\dots,n\}$ is the vertex set, $D\subset V\times V$ are the
  directed edges, and $B\subset V\times V$ are the bidirected edges.
  We require that there be no self-loops, so $(v,v)\not\in D,B$ for
  all $v\in V$. If $(v,w)\in D$, we will write $v\to w\in G$ and if
  $(v,w)\in B$, we will write $v\bi w\in G$. As bidirected edges are
  symmetric we will also require that $B$ is symmetric, so that
  $(v,w)\in B \iff (w,v)\in B$.
\end{defn}

Let $v$ and $w$ be two vertices of a mixed graph $G=(V,D,B)$.  A
\emph{path} from $v$ to $w$ is any sequence of edges from $D$ or $B$
beginning at $v$ and ending at $w$.  Here, we allow that directed
edges be traversed against their natural direction (i.e., from head to
tail).  We also allow repeated vertices on a path.  Sometimes, such
paths are referred to as walks or also semi-walks.  A path from $v$ to
$w$ is \emph{directed} if all of its edges are directed and point in
the same direction, away from $v$ and towards $w$.

\begin{defn}[Treks and half-treks]
  (a)
  A path $\pi$ from a \emph{source} $v$ to a \emph{target} $w$ is a
  \emph{trek} if it has no colliding arrowheads, that is, $\pi$ is of
  the form
  \begin{align*}
    &v^{L}_{l}\leftarrow v^{L}_{l-1} \leftarrow \dots\leftarrow
    v^{L}_{0} \longleftrightarrow v^{R}_{0} \to v^{R}_{1} \to \dots
    \to v^{R}_{r-1} \to v^{R}_{r} \quad\text{or}\\
    &v^{L}_{l}\leftarrow v^{L}_{l-1} \leftarrow \dots\leftarrow v^{L}_{1} \leftarrow v^T \to v^{R}_{1} \to \dots \to v^{R}_{r-1} \to v^{R}_{r},
  \end{align*}
  where $v^L_l = v$, $v^R_r = w$, and $v^T$ is the \emph{top} node.
  Each trek $\pi$ has a left-hand side $\Left(\pi)$ and a right-hand
  side $\Right(\pi)$.  In the former case,
  $\Left(\pi) = \{v^{L}_{0},\dots,v^{L}_{l}\}$ and
  $\Right(\pi) = \{v^{R}_{0},\dots,v^{R}_{r}\}$.  In the latter case,
  $\Left(\pi) = \{v^T, v^{L}_{1},\dots,v^{L}_{l}\}$ and
  $\Right(\pi) = \{v^T, v^{R}_{1},\dots,v^{R}_{r}\}$, with $v^T$ a part
  of both sides.

  (b) A trek $\pi$ is a \emph{half-trek} if $|\Left(\pi)| =
  1$.  In this case $\pi$ is of the form
  \begin{align*}
    &v^{L}_{0} \longleftrightarrow v^{R}_{0} \to v^{R}_{1} \to \dots
      \to v^{R}_{r-1} \to v^{R}_{r} \quad\text{or}\quad
      v^T \to v^{R}_{1} \to \dots \to v^{R}_{r-1} \to v^{R}_{r}.
  \end{align*}
  In particular, a half-trek from $v$ to $w$ is a trek from $v$ to $w$ which
  is either empty, begins with a bidirected edge, or begins with a directed
  edge pointing away from $v$.
\end{defn}

Some terminology is needed to reference the local neighborhood
structure of a vertex $v$.  For the directed part $(V,D)$,
it is standard to define the set of \emph{parents} and the set of
\emph{descendents} of $v$ as
\begin{align*}
  \pa(v) &= \{w\in V:w\to v\in G\} , \\
  \des(v) &= \{w\in V: \exists\text{ a non-empty directed path from
           $v$ to $w$ in $G$}\},
\end{align*}
respectively.  The nodes incident to a bidirected edge can be thought
of as having a common (latent) parent and thus we refer to the
bidirected neighbors as \emph{siblings} and define
\begin{align*}
\sib(v) &= \{w\in V:w\bi v\in G\}.
\end{align*}
 Finally, we denote the sets of nodes that are \emph{trek reachable}
 or \emph{half-trek reachable}  from $v$ by
\begin{align*}
\tr(v)  &= \{w\in V: \exists\text{ a non-empty trek from $v$ to $w$ in $G$}\}, \\
\htr(v)  &= \{w\in V:
\exists\text{ a non-empty half-trek from $v$ to $w$ in
  $G$}\}.
\end{align*}

Two sets of matrices may be associated with a given mixed graph
$G=(V,D,B)$.  First, $\bR_{\mathrm{reg}}^D$ is the set of real
$n\times n$ matrices $\Lambda=(\lambda_{vw})$ with support $D$, i.e.,
those matrices $\Lambda$ with $\lambda_{vw}\not=0$ implying
$v\to w\in G$ and for which $I - \Lambda$ invertible.  Second,
$\mathit{PD}(B)$ is the set of positive definite matrices with support
$B$, i.e., if $v\not=w$, then $\omega_{vw}\not=0$ implies
$v\bi w\in G$.  Based on~(\ref{eq:sigma}), the distributions in the
L-SEM given by $G$ have a covariance matrix $\Sigma$ that is
parameterized by the map
\begin{align}
  \phi_G:(\Lambda,\Omega) \mapsto (I-\Lambda)^{-T}\Omega(I-\Lambda)^{-1}
\end{align}
with domain $\Theta := \bR_{\mathrm{reg}}^D\times\mathit{PD}(B)$.  

\begin{rem}
  Our focus is solely on covariance matrices.  Indeed, in
  the traditional case where the errors $\epsilon$
  in~(\ref{eq:sem_eq}) follow a multivariate normal distribution the
  covariance matrix contains all available information about the
  parameters $(\Lambda,\Omega)$.
\end{rem}

Subsequently, the matrices $\Lambda,\Omega$ and $\Sigma$ will also be
regarded as matrices of indeterminants.  The entries of
$(I-\Lambda)^{-1} = I + \sum_{k=1}^\infty \Lambda^k$ may then be
interpreted as formal power series.  Let $\Lambda$ and $\Omega$ be
matrices of indeterminants with zero pattern corresponding to $G$.
Then $\Sigma=\phi_G(\Lambda,\Omega)$ has entries that are formal power
series whose form is described by the Trek Rule of \citet{wright1921},
see also \citet*{spirtes2000}. The Trek rule states that for every
$v,w\in V$ the corresponding entry of $\phi_G(\Lambda, \Omega)$ is the
sum of all trek monomials corresponding to treks from $v$ to $w$.

\begin{defn}[Trek Monomial]
  Let $\cT(v,w)$ be the set of all treks from $v$ to $w$ in $G$.  If
  $\pi\in \cT(v,w)$ contains no bidirected edge and has top node $z$,
  its \emph{trek monomial} is defined as
  \begin{align*}
    \pi(\Lambda,\Omega) = \omega_{zz}\prod_{x\to y\in \pi} \lambda_{xy}.
  \end{align*}
  If $\pi$ contains a bidirected edge connecting $u,z\in V$, then its
  trek monomial is
  \begin{align*}
    \pi(\Lambda,\Omega) = \omega_{uz}\prod_{x\to y\in \pi} \lambda_{xy}.
  \end{align*}
\end{defn}

\begin{prop}[Trek Rule] \label{prop:trek-rule} The covariance matrix
  $\Sigma = \phi_G(\Lambda,\Omega)$ corresponding to a mixed graph $G$
  satisfies
  \begin{align*}
    \Sigma_{vw} = \sum_{\pi\in \cT(v,w)} \pi(\Lambda,\Omega), \quad
    v,w\in V.
  \end{align*}
\end{prop}




\subsection{Generic Identifiability} \label{sec:htc}

We now formally introduce our problem of interest and review some of
the prior work our results build on.  We recall that an
\emph{algebraic set} is the zero-set of a collection of polynomials.
An algebraic set that is a proper subset of Euclidean space has
measure zero; see, e.g., the lemma in \cite{okamoto:1973}.

\begin{defn}[Generic Identifiability]
  (a) The model given by a mixed graph $G$ is \emph{generically
    identifiable} if there exists a proper algebraic subset
  $A\subset \Theta$ such that the fiber
  $\cF(\Lambda,\Omega) := \phi_G^{-1}(\{\phi_G(\Lambda,\Omega)\})$ is
  a singleton set, that is, it satisfies
  \begin{align*}
    \cF(\Lambda,\Omega) = \{(\Lambda,\Omega)\}
  \end{align*}
  for all $(\Lambda,\Omega)\in \Theta\setminus A$.  In this case we
  will say, for simplicity, that $G$ is generically identifiable.
  
  (b) Let $\proj_{v\to w}$ be the projection
  $(\Lambda,\Omega)\mapsto \lambda_{vw}$ for $v\to w\in G$.  We say
  that the edge coefficient $\lambda_{vw}$ is \emph{generically
    identifiable} if there exists a proper algebraic subset
  $A\subset \Theta$ such that
  $\proj_{v\to w}(\cF(\Lambda,\Omega))=\{\lambda_{vw}\}$ for all
  $(\Lambda,\Omega)\in \Theta\setminus A$.  In this case, we will say
  that the edge $v\to w$ is generically identifiable.
\end{defn}

In all examples we know of, if generic identifiability holds,
then the parameters can in fact be recovered using rational formulas.

\begin{defn}[Rational Identifiability]
  (a) A mixed graph $G$, or rather the model it defines, is
  \emph{rationally identifiable} if there exists a rational map $\psi$
  and a proper algebraic subset $A\subset \Theta$ such that
  $\psi\circ \phi_G$ is the identity on $\Theta\setminus A$.

  (b) An edge $v\to w\in G$, or rather the coefficient $\lambda_{vw}$,
  is \emph{rationally identifiable} if there exists a rational
  function $\psi$ and a proper algebraic subset $A\subset \Theta$ such
  that $\psi\circ \phi_G(\Lambda,\Omega)=\lambda_{vw}$ for all
  $(\Lambda,\Omega)\in \Theta\setminus A$.
\end{defn}

We now introduce the half-trek criterion (HTC) of \cite{halftrek}. We
generalize this criterion in Section \ref{sec:edgewise-id}.

\begin{defn}[Trek and Half-Trek Systems]
  Let $\Pi = \{\pi_1,\dots,\pi_m\}$ be a collection of treks in
  $G$ and let $S,T$ be the set of sources and targets of the $\pi_i$
  respectively.  Then we say that $\Pi$ is a \emph{system of treks}
  from $S$ to $T$. If each $\pi_i$ is a half-trek, then $\Pi$ is a
  \emph{system of half-treks}. A collection
  $\Pi=\{\pi_1,\dots,\pi_m\}$ of treks is said to have \emph{no sided
    intersection} if
  \begin{align*}
    \Left(\pi_i)\cap\Left(\pi_j) =\emptyset = \Right(\pi_i)\cap\Right(\pi_j), \ \forall i\not=j.
  \end{align*}
\end{defn}

As our focus will be on the identification of individual edges in $G$ we do not state the identifiability result of \citet{halftrek} in its usual form, instead we present a slightly modified version which is easily seen to be implied by the proof of Theorem 1 in \cite{halftrek}.

\begin{defn}
  A set of nodes $Y\subset V$ satisfies the \emph{half-trek criterion} with respect to a vertex $v\in V$ if
  \begin{enumerate}[(i)]
  \item $|Y|=|\pa(v)|$,
  \item $Y\cap (\{v\}\cup \sib(v)) = \emptyset$, and
  \item there is a system of half-treks with no sided intersection from $Y$ to $\pa(v)$.
  \end{enumerate}
\end{defn}

\begin{thm}[HTC-identifiability] \label{thm:htc-id}
  Suppose that in the mixed graph $G=(V,D,B)$ the set $Y\subset V$
  satisfies the half-trek criterion with respect to $v\in V$.  If all
  directed edges $u\to y\in G$ with head $y\in \htr(v)\cap Y$ are
  generically (rationally) identifiable, then all directed edges with
  $v$ as a head are generically (rationally) identifiable.
\end{thm}

The sufficient condition for rational identifiability of $G$ in
\cite{halftrek} is obtained through iterative application of Theorem
\ref{thm:htc-id}.

\section{Trek Separation and Identification by Ratios of Determinants}\label{sec:trek-sep}

Let $\Lambda$ and $\Omega$ be matrices of indeterminants corresponding
to a mixed graph $G = (V,D,B)$ as specified in
Section~\ref{sec:mixed-graphs}.  Let $S,T\subset V$, and let
$\Sigma_{S,T}$ be the submatrix of
$\Sigma=\phi_G(\Lambda,\Omega)\in\mathbb{R}^{n\times n}$ obtained by
retaining only the rows and columns indexed by $S$ and $T$,
respectively.  The (generic) rank of such a submatrix $\Sigma_{S,T}$
can be completely characterized by considering the trek systems
between the vertices in $S$ and $T$. The formal statement of this
result follows.

\begin{defn}[t-separation]
  A pair of sets $(L,R)$ with $L,R\subset V$ \emph{t-separates} the
  sets $S,T\subset V$ if every trek between a vertex $s\in S$ and a
  vertex $t\in T$ intersects $L$ on the left or $R$ on the right.
\end{defn}

In this definition, the symbols $L$ and $R$ are chosen to suggest left
and right.  Similarly, $S$ and $T$ are chosen to indicate sources and
targets, respectively.

\begin{thm}[\cite{sullivant2010}, \cite{draisma2013}] \label{thm:t-sep}
  The submatrix $\Sigma_{S,T}$ has generic rank $\leq r$ if and only if there exist sets $L,R\subset V$ with $|L|+|R|\leq r$ such that $(L,R)$ t-separates $S$ and $T$.
\end{thm}

Theorem 2.7 of \cite{sullivant2010} established this result for
acyclic mixed graphs while \cite{draisma2013} extended the result to
all mixed graphs and even gave an explicit representation of the
rational form of the subdeterminant $|\Sigma_{S,T}|$, for $|S|=|T|$. An immediate
corollary to the above theorem, considering the proof of Theorem 2.17
in \cite{sullivant2010}, rephrases its statement in terms of maximum
flows in a special graph. For an introduction to maximum flow, and the
well-known Max-flow Min-cut Theorem, see the book by
\cite{cormen2009}. Note that standard max-flow min-cut framework does not
allow vertices to have maximum capacities or for there to be multiple
sources and targets, introducing these modifications is, however, trivial
and the resulting theorem is sometimes called the Generalized Max-flow
Min-cut Theorem.

\begin{cor} \label{cor:t-sep} Let $G_{\mathrm{flow}} = (V_f, D_f)$ be
  the directed graph with $V_f = \{1,\dots,n\}\cup\{1',\dots,n'\}$ and
  $D_f$ containing the following edges:
  \begin{align}
     i\to j &\ \text{ if }   j\to i\in G, \label{eq:rev-dir}\\
     i\to i' &\ \text{ for all } i\in V, \label{eq:node-to-copy}\\
    i \to j' &\ \text{ if } i\bi j \in G, \text{ and} \label{eq:bi-edges}\\
     i'\to j' &\ \text{ if } i\to j\in G. \label{eq:dir}
  \end{align}
  Turn $G_{\mathrm{flow}}$ into a network by giving all vertices and
  edges capacity 1. Let
  $S=\{s_1,\dots,s_k\}, T=\{t_1,\dots,t_m\}\subset V$.  Then
  $\Sigma_{S,T}$ has generic rank $r$ if and only if the max-flow from
  $s_1,\dots,s_k$ to $t_1',\dots,t_m'$ in $G_{\mathrm{flow}}$ is $r$.
\end{cor}

\begin{proof}
Add vertices $u,v$, with infinite capacity, to the graph $G_{flow}$ along with edges, all with capacity 1, $u\to s_i$, for $1 \leq i \leq k$, and $t'_j\to v$, for $1 \leq j \leq m$. Let $L,R$ be such that they $t$-separate the sets $S,T$ and $|L| + |R|$ is minimal. By Theorem \ref{thm:t-sep}, $\Sigma_{S,T}$ has rank $|L|+|R|$ generically. Note that $L \cup R’$ gives the minimal size $s-t$ cut (of size $|L|+|R|$). By the (generalized) Max-flow Min-cut theorem the max-flow from $u$ to $v$ is $|L|+|R|$, and it is also the max flow from $s_1,\ldots,s_k$ to $t'_1,\ldots,t'_m$. Hence $\Sigma_{S,T}$ has generic rank equal to the found max-flow.
\end{proof}

Note that the maximum flow between vertex sets in a graph can be
computed in polynomial time. Indeed, in our case, the conditions of
Corollary \ref{cor:t-sep} can be checked in $O(|V|^2 \max\{m,k\})$
time \citep[page 725]{cormen2009}. As the following example shows,
Corollary \ref{cor:t-sep} can be used to find determinantal
constraints on $\Sigma$.   These constraints can then be leveraged to
identify edges in $G$.

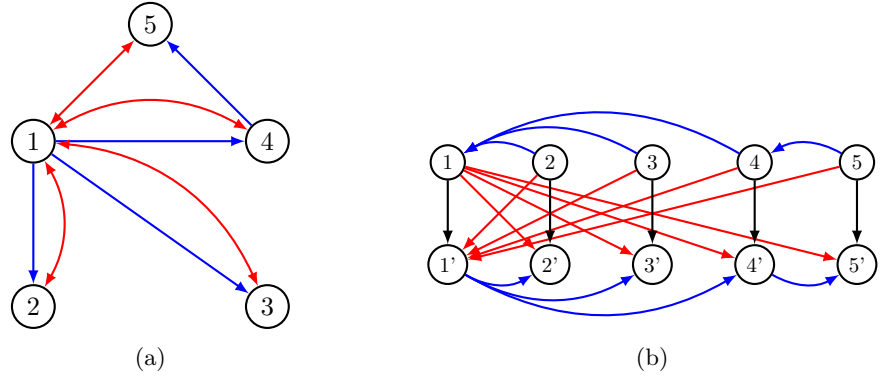
\begin{figure}[t]
  \begin{subfigure}[b]{0.4\textwidth}
    \centering
    \tikzset{
      every node/.style={circle, inner sep=1mm, minimum size=0.55cm, draw, thick, black, fill=white, text=black},
      every path/.style={thick}
    }
    \begin{tikzpicture}[align=center,node distance=2.2cm]
      \node [] (5) {5};
      \node [] (1) [below left of=5] {1};
      \node [] (4) [below right of=5]    {4};
      \node [] (2) [below of=1]    {2};
      \node [] (3) [below of=4]    {3};

      \draw[blue] [-latex] (1) edge (2);
      \draw[blue] [-latex] (1) edge (3);
      \draw[blue] [-latex] (1) edge (4);
      \draw[blue] [-latex] (4) edge (5);
      \draw[red] [latex-latex, bend left] (1) edge (2);
      \draw[red] [latex-latex, bend left] (1) edge (3);
      \draw[red] [latex-latex, bend left] (1) edge (4);
      \draw[red] [latex-latex] (1) edge (5);
    \end{tikzpicture}
    \caption{} \label{fig:htc-u}
  \end{subfigure}
  ~ \hspace{0.04\textwidth}
  \begin{subfigure}[b]{0.55\textwidth}
    \centering
    \tikzset{
      every node/.style={circle, inner sep=1mm, minimum size=0.55cm, draw, thick, black, fill=white, text=black, scale=0.8},
      every path/.style={thick}
    }
    \begin{tikzpicture}[align=center,node distance=1.7cm]
      
      \node [] (1) {1};
      \node [] (2) [right of=1] {2};
      \node [] (3) [right of=2] {3};
      \node [] (4) [right of=3] {4};
      \node [] (5) [right of=4] {5};
      
      \node [] (1') [below of=1] {1'};
      \node [] (2') [below of=2] {2'};
      \node [] (3') [below of=3] {3'};
      \node [] (4') [below of=4] {4'};
      \node [] (5') [below of=5] {5'};
      
      \draw[blue] [-latex, bend right] (2) edge (1);
      \draw[blue] [-latex, bend right] (3) edge (1);
      \draw[blue] [-latex, bend right] (4) edge (1);
      \draw[blue] [-latex, bend right] (5) edge (4);

      \draw[red] [-latex] (1) edge (2');
      \draw[red] [-latex] (1) edge (3');
      \draw[red] [-latex] (1) edge (4');
      \draw[red] [-latex] (1) edge (5');	
      \draw[red] [-latex] (2) edge (1');
      \draw[red] [-latex] (3) edge (1');
      \draw[red] [-latex] (4) edge (1');
      \draw[red] [-latex] (5) edge (1');
      
      \draw[black] [-latex] (1) edge (1');
      \draw[black] [-latex] (2) edge (2');
      \draw[black] [-latex] (3) edge (3');
      \draw[black] [-latex] (4) edge (4');
      \draw[black] [-latex] (5) edge (5');
      
      \draw[blue] [-latex, bend right] (1') edge (2');
      \draw[blue] [-latex, bend right] (1') edge (3');
      \draw[blue] [-latex, bend right] (1') edge (4');
      \draw[blue] [-latex, bend right] (4') edge (5');
    \end{tikzpicture}
    \caption{} \label{fig:htc-u-flow}
  \end{subfigure}
  
  \caption{(a) A graph $G$ that is generically identifiable but
    for which the HTC fails to identify any
    coefficients. (b) The corresponding flow graph
    $G_{\mathrm{flow}}$, black edges correspond to
    \eqref{eq:node-to-copy}, red edges to \eqref{eq:bi-edges}, and
    blue edges to \eqref{eq:rev-dir} and
    \eqref{eq:dir}.}\label{fig:htc-u-example} 
\end{figure}

\begin{exmp} \label{exmp:t-sep-1} Consider the mixed graph
  $G = (V,D,B)$ in Figure \ref{fig:htc-u}, which is taken from
  Fig.~3c in \cite{halftrek}.  The corresponding flow network
  $G_{\mathrm{flow}}$ is shown in Figure \ref{fig:htc-u-flow}.  From
  Gr\"obner basis computations, $G$ is known to be rationally
  identifiable but the half-trek criterion fails to certify that any
  edge of $G$ is generically identifiable. Let $S = \{1,2,4\}$ and
  $T = \{1,3,5\}$.  Corollary \ref{cor:t-sep} implies that
  $\Sigma_{S,T}$ has generically full rank as there is a flow of size
  3 from $S$ to $T'=\{1',3',5'\}$ in $G_{\mathrm{flow}}$, via the
  paths $1\to 3'$, $2\to 1'$, and $4\to 5'$. Now suppose that we
  remove the $4\to 5$ edge from $G$, call the resulting graph $\bar G$,
  and let $\bar\Sigma$ be the covariance matrix corresponding to
  $\bar G$. Then one may check that the max-flow from $S$ to $T'$ in
  $\bar G_{\mathrm{flow}}$ is $\leq 2$. Thus
  $|\bar\Sigma_{\{1,2,4\},\{1,3,5\}}| = 0$ where $|\cdot |$ denotes the
  determinant. Now note that $\lambda_{45}\sigma_{14}$ is the sum of
  all monomials given by treks from 1 to 5 that end in the edge
  $\lambda_{45}$.  Hence, $\sigma_{15} - \lambda_{45}\sigma_{14}$ is
  obtained by summing over all treks from 1 to 5 that do not end in
  the edge $4\to 5$.  But in our graph this is just the sum over treks
  from 1 to 5 that do not use the edge $4\to 5$ at all.  Therefore,
  $\bar\sigma_{15} = \sigma_{15} - \lambda_{45}\sigma_{14}$.  Similarly, it
  is straightforward to check that
  \begin{align}
    \bar\Sigma_{\{1,2,4\},\{1,3,5\}} = \left(\begin{matrix} \sigma_{11} & \sigma_{13} & \sigma_{15} -  \lambda_{45}\sigma_{14}\\
        \sigma_{21} & \sigma_{23} & \sigma_{25} -  \lambda_{45}\sigma_{24}\\
        \sigma_{41} & \sigma_{43} & \sigma_{45} -  \lambda_{45}\sigma_{44}
      \end{matrix} \right).\label{eq:equating-sub-mats}
  \end{align}
  By the multilinearity of the determinant, we deduce that
  \begin{align*}
    0 = |\bar\Sigma_{\{1,2,4\},\{1,3,5\}}| 
      &= 
        \left|\begin{matrix} \sigma_{11} & \sigma_{13} & \sigma_{15}\\
            \sigma_{21} & \sigma_{23} & \sigma_{25} \\
            \sigma_{41} & \sigma_{43} & \sigma_{45}
          \end{matrix} \right| -
                                        \lambda_{45}\left|\begin{matrix} \sigma_{11} & \sigma_{13} &  \sigma_{14}\\
                                            \sigma_{21} & \sigma_{23} & \sigma_{24}\\
                                            \sigma_{41} & \sigma_{43} &  \sigma_{44}
                                          \end{matrix} \right| \\
      &= |\Sigma_{\{1,2,4\},\{1,3,5\}}| - \lambda_{45}|\Sigma_{\{1,2,4\},\{1,3,4\}}|.
  \end{align*}
  Applying Corollary \ref{cor:t-sep} a final time, we recognize that
  $|\Sigma_{\{1,2,4\},\{1,3,4\}}|$ is generically non-zero and, thus,
  the equation 
  \[
    \lambda_{45} =
    \frac{|\Sigma_{\{1,2,4\},\{1,3,5\}}|}{|\Sigma_{\{1,2,4\},\{1,3,4\}}|}
  \] generically and rationally identifies $\lambda_{45}$.  In this
  case, the same strategy can be used to identify the edges
  $1\to 2$ and $1\to 3$ (but not $1\to 4$) in $G$.
\end{exmp}

In the above example, there is a correspondence between trek systems
in $G$ and trek systems in $\bar G$, the graph that has the edge to be
identified  removed.  This allowed us to leverage Corollary
\ref{cor:t-sep} directly to show that \eqref{eq:equating-sub-mats} has
determinant 0. Such a correspondence cannot always be obtained but
exists in the following case.

\begin{thm}\label{thm:t-sep-id-simple}
  Let $G = (V,D,B)$ be a mixed graph. Let $w_0\to v$ be an edge in
  $G$, and suppose that the edges $w_1\to v,\dots,w_\ell\to v \in G$
  are known to be generically (rationally) identifiable.  Let $\bar G$
  be the subgraph of $G$ with the edges
  $w_0 \rightarrow v,\dots,w_\ell\to v\in G$ removed. Suppose there
  are sets $S\subset V\setminus\{v\}$,
  $T\subset V \setminus \{v,w_0\}$ with $|S| = |T|+1 = k$ such that:
  \begin{enumerate}[(a)]
  \item $\des(v)\cap (S\cup T \cup\{v\}) = \emptyset$,
  \item the max-flow from $S$ to $T'\cup\{w_0'\}$ in $G_{\mathrm{flow}}$ equals $k$, and
  \item the max-flow from $S$ to $T'\cup\{v'\}$ in $\bar
    G_{\mathrm{flow}}$ is smaller than $k$.
  \end{enumerate}
  Then $w_0\to v$ is generically (rationally) identifiable by the equation
  \begin{align}\label{lambdaFormula}
    \lambda_{w_0v} = \frac{|\Sigma_{S, T \cup\{v\}}| - \sum_{i=1}^\ell\lambda_{w_iv}|\Sigma_{S, T \cup\{w_i\}}|}{| \Sigma_{S, T \cup\{w_0\}}|}.
  \end{align}
\end{thm}

\begin{proof}
  Let $\Sigma$ and $\bar\Sigma$ be the covariance matrices
  corresponding to $G$ and $\bar G$, respectively. Since
  $\des(v)\cap (S\cup T \cup\{v\}) = \emptyset$, we have that
  $\sigma_{st} = \bar\sigma_{st}$ for all $s\in S$ and $t\in T$.  This holds because if a trek from $s$ to $t$ uses an edge $w_i\to v$ then
  either $s\in \{v\}\cup \des(v)$ or $t\in \{v\}\cup \des(v)$,
  violating our assumptions.

  Now let $s\in S$ and $0\leq i\leq \ell$. Suppose that $\pi$ is a
  trek from $s$ to $v$ that uses the edge $w_i\to v$. Then since
  $s\not\in \{v\}\cup \des(v)$ we must have that $w_i\to v$ is used
  only on the right-hand side of $\pi$.  With $v\not\in \des(v)$ it
  follows that $w_i\to v$ is the last edge used in the trek because
  $\pi$ may only use directed edges after using $w_i\to v$ and must
  end at $v$.  Hence, all treks from $s$ to $v$ which use $w_i\to v$
  must have this edge as their last edge on the right.  But
  $\sigma_{sw_i}\lambda_{w_iv}$ is obtained by summing over all treks
  from $s$ to $v$ which end in the edge $w_i\to v$ and, thus,
  $\sigma_{sv} - \sigma_{sw_i}\lambda_{w_iv}$ is the sum of the
  monomials for
  all treks from $s$ to $v$ that do not use the $w_i\to v$ edge at all.

  As the above argument holds for all $0\leq i\leq \ell$, it follows
  that
  $\bar\sigma_{sv} = \sigma_{sv} -
  \sum_{i=0}^k\sigma_{sw_i}\lambda_{w_iv}$. Since this is true for all
  $s\in S$ it follows, similarly as in Example \ref{exmp:t-sep-1},
  that
  \begin{align*}
    |\bar\Sigma_{S,T\cup \{v\}}| = |\Sigma_{S,T\cup\{v\}}| - \sum_{i=0}^k\lambda_{w_iv}|\Sigma_{S,T\cup\{w_i\}}|.
  \end{align*}
  Using assumption (c) and applying Corollary \ref{cor:t-sep}, we have
  $|\bar\Sigma_{S,T\cup \{v\}}| = 0$.  Similarly, by assumption (b),
  $|\Sigma_{S,T\cup\{w_0\}}|\not = 0$ generically. The desired result follows.
\end{proof}

\begin{rem}
  Theorem~\ref{thm:t-sep-id-simple} generalizes the ideas underlying
  instrumental variable methods such as those discussed in
  \cite{Brito02}.  Indeed, this prior work uses d-separation as
  opposed to t-separation.  D-separation characterizes conditional
  independence which in the present context corresponds to the
  vanishing of particular almost principal determinants of the
  covariance matrix.  In contrast,
  Theorem~\ref{thm:t-sep-id-simple} allows us to leverage arbitrary
  determinantal relations; compare \cite{sullivant2010}.  The graph
  in Figure \ref{fig:htc-u} is an example in which
  d-separation and traditional instrumental variable techniques cannot
  explain the rational identifiability of the coefficient for edge
  $4\to 5$.
\end{rem}

While assumption (a) in the above Theorem allows for the easy
application of Corollary \ref{cor:t-sep}, this assumption can be
relaxed by generalizing one direction of Corollary \ref{cor:t-sep}. We
state this generalization as the following lemma, which is concerned
with asymmetric treatment of edges that appear on the left versus
right-hand side of treks.  The lemma's proof is deferred to Appendix
\ref{app:proof-0-det-generalization}.

\begin{lem}\label{lem:0-det-generalization}
  Let $G=(V,D,B)$ be a mixed graph, and let $\Lambda=(\lambda_{uv})$ and $\Omega$ be
  the matrices of indeterminants corresponding to the directed and the
  bidirected part of $G$, respectively.  Let $D_L,D_R\subset D$
  and define $n\times n$ matrices $\Lambda^L$ and $\Lambda^R$ with
  \begin{align*}
    \Lambda^L_{uv} &= 
                   \begin{cases}
                     \lambda_{uv}  & \mbox{if } (u,v)\in D_L, \\
                     0 & \mbox{otherwise},
                   \end{cases} \text{ and }\\
    \Lambda^R_{uv} &= 
                   \begin{cases}
                     \lambda_{uv}  & \mbox{if } (u,v)\in D_R, \\
                     0 & \mbox{otherwise}.
                   \end{cases} 
  \end{align*}
  Define a network $G^*_{\mathrm{flow}} = (V^*,D^*)$ with vertex set
  $V^* = \{1,\dots,n\}\cup\{1',\dots,n'\}$, edge set $D^*$ containing 
  \begin{align}
    i\to j   & \ \text{ if } (j,i)\in D_L, \label{eq:rev-dir-2}\\
    i\to i'  & \ \text{ for all } i\in V, \label{eq:node-to-copy-2}\\
    i \to j' & \ \text{ if } (i,j) \in B, \label{eq:bi-edges-2}\\
    i'\to j' & \ \text{ if } (i, j)\in D_R, \text{ and} \label{eq:dir-2}
  \end{align}
  with all edges and vertices having capacity 1.  Let
  $\Gamma = (I-\Lambda^L)^{-T}\Omega(I-\Lambda^R)^{-1}$.  Then, for any
  $S,T\subset V$ with $|S|=|T|=k$, we have that $|\Gamma_{S,T}|=0$ if
  the max-flow from $S$ to $T'$ in $G^*_{\mathrm{flow}}$ is $<k$.
\end{lem}

We may now state our more general result.

\begin{thm}\label{thm:t-sep-id-complex}
  Let $G = (V,D,B)$ be a mixed graph, $w_0\to v\in G$, and suppose that the edges $w_1\to v,\dots,w_\ell\to v\in G$ are known to be generically (rationally) identifiable. Let $G^*_{\mathrm{flow}}$ be $G_{\mathrm{flow}}$ with the edges $w_0'\to v',\dots,w_\ell'\to v'$ removed. Suppose there are sets $S\subset V$ and $T \subset V \setminus \{v,w_0\}$ such that $|S| = |T|+1 = k$ and
  \begin{enumerate}[(a)]
  \item $\des(v)\cap (T \cup\{v\}) = \emptyset$,
  \item the max-flow from $S$ to $T'\cup\{w_0'\}$ in $G_{\mathrm{flow}}$ equals $k$, and
  \item the max-flow from $S$ to $T'\cup \{v'\}$ in $G^*_{\mathrm{flow}}$ is $<k$.
  \end{enumerate}
  Then $w_0\to v$ is rationally identifiable by the equation
  \begin{align}
    \label{eq:div-formula}
    \lambda_{w_0v} = \frac{|\Sigma_{S, T \cup\{v\}}| - \sum_{i=1}^\ell\lambda_{w_iv}|\Sigma_{S, T \cup\{w_i\}}|}{| \Sigma_{S, T \cup\{w_0\}}|}.
  \end{align}
\end{thm}

\begin{proof}
  By assumption (b) and Corollary \ref{cor:t-sep},
  $|\Sigma_{S, T \cup\{w_0\}}|$ is generically non-zero.  Therefore,
  equation~(\ref{eq:div-formula}) holds if
  \[
    |\Sigma_{S, T \cup\{v\}}| -
    \sum_{i=0}^\ell\lambda_{w_iv}|\Sigma_{S, T \cup\{w_i\}}| = 0.
  \]
  To show this we note that, by the multilinearity of the determinant,
  we have
  \begin{align*}
    |\Sigma_{S, T \cup\{v\}}| - \sum_{i=0}^\ell\lambda_{w_iv}|\Sigma_{S, T \cup\{w_i\}}| = 
    \left|\begin{matrix}
	\sigma_{s_1t_1} & \hdots & \sigma_{s_1t_{k-1}} & \sigma_{s_1v} - \sum_{i=0}^\ell\lambda_{w_iv}\sigma_{s_1w_i} \\
	\sigma_{s_2t_1} & \hdots & \sigma_{s_2t_{k-1}} & \sigma_{s_2v} - \sum_{i=0}^\ell\lambda_{w_iv}\sigma_{s_2w_i} \\
	\vdots & \ddots & \vdots &  \vdots \\
	\sigma_{s_kt_1} & \hdots & \sigma_{s_kt_{k-1}} & \sigma_{s_kv} - \sum_{i=0}^\ell\lambda_{w_iv}\sigma_{s_kw_i} \\
      \end{matrix} \right|.
  \end{align*}
  Write $\Gamma$ for the matrix that appears on the right-hand side of
  this equation.  
  
  Consider any two indices $i$ and $j$ with $1\leq i\leq k$ and
  $1\leq j\leq k-1$.  If a trek from $s_i$ to $t_j$ uses one of the
  edges $w_m\to v$, for $0\leq m\leq \ell$, on its right-hand side
  then $t_j\in\des(v)$, a contradiction since
  $\des(v)\cap T = \emptyset$ by assumption (a).  Similarly, since
  $v\not\in \des(v)$ the difference
  $\sigma_{s_iv} - \sum_{j=0}^\ell\lambda_{w_jv}\sigma_{s_iw_j}$ is
  obtained by summing the monomials for treks between $s_i$ and $v$
  which do not use any edge $w_j\to v$ on their right side. From this
  we may write
  \begin{align*}
    \Gamma = ((I - \Lambda)^{-T}\Omega(I-\Lambda')^{-1})_{S,T\cup\{v\}}
  \end{align*}
  where $\Lambda'$ equals $\Lambda$ but with its $(w_j,v)$, $0\leq
  j\leq \ell$, entries set to 0. The fact that $|\Gamma|=0$ under
  assumption (c) is the content of Lemma
  \ref{lem:0-det-generalization} (where we take $\Lambda^L=\Lambda$
  and $\Lambda^R= \Lambda'$). Given this lemma our desired result then follows.
\end{proof}

Clearly Theorem \ref{thm:t-sep-id-complex} can be applied whenever
Theorem \ref{thm:t-sep-id-simple} can.  Moreover, as the next example
shows, there are cases in which Theorem \ref{thm:t-sep-id-complex} can
be used while Theorem \ref{thm:t-sep-id-simple} cannot.

\begin{exmp}
  Let $G = (V,D,B)$ be the mixed graph from Figure
  \ref{fig:more-applicable-example}.  Take $S=\{3,5\}$ and $T=\{4\}$.
  Then Theorem \ref{thm:t-sep-id-complex} implies that $\lambda_{12}$
  is rationally identifiable. Theorem \ref{thm:t-sep-id-simple} cannot
  be applied in this case as $S\cap \des(2)\not=\emptyset$.
\end{exmp}

\begin{figure}[t]
  \centering
  \tikzset{
    every node/.style={circle, inner sep=1mm, minimum size=0.55cm, draw, thick, black, fill=white, text=black},
    every path/.style={thick}
  }
  \begin{tikzpicture}[align=center,node distance=1.5cm]
    \node [] (1) {1};
    \node [] (2) [right of=1] {2};
    \node [] (3) [right of=2]    {3};
    \node [] (4) [above of=3]    {4};
    \node [] (5) [right of=3]    {5};

    \draw[blue] [-latex] (1) edge (2);
    \draw[blue] [-latex] (2) edge (3);
    \draw[blue] [-latex] (4) edge (3);
    \draw[blue] [-latex] (3) edge (5);
    
    \draw[red] [latex-latex, bend left] (1) edge (2);
    \draw[red] [latex-latex, bend left] (3) edge (5);
    \draw[red] [latex-latex, bend right] (1) edge (5);
  \end{tikzpicture}
  \caption{A graph for which Theorem \ref{thm:t-sep-id-complex} can be used to to certify that the edge $\lambda_{12}$ is identifiable when Theorem \ref{thm:t-sep-id-simple} cannot.} \label{fig:more-applicable-example}
\end{figure}
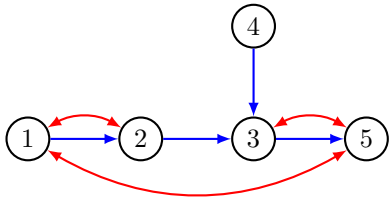

For a fixed choice of $S$ and $T$, the conditions (a)-(c) in Theorem
\ref{thm:t-sep-id-complex} can be verified in polynomial time.
Indeed, conditions (b) and (c) involve only max-flow
computations that take $O(|V|^3)$ time in general.  Condition (a) can
be checked by computing the descendants of $v$, which can be done with
any $O(|D|)$ graph traversal algorithm (e.g., depth first search, see \citet{cormen2009}), and
then computing the intersection between the descendants and
$T\cup\{v\}$ which can be done in $O(|V|\log|V|)$ time.

In order to apply Theorem \ref{thm:t-sep-id-complex} algorithmically,
however, we have to consider all possible subsets $S\subset V$,
$T\subset V\setminus\{v,w_0\}$ and check our condition for each pair.
Naively done this operation takes exponential time.  It remains an
interesting problem for further study to determine whether or not the
problem of finding suitable sets $S$ and $T$ is NP-hard.  We note that
a similar problem arises for instrumental variables/d-separation,
where \cite{vanderzander:2015} were able to give a polynomial time
algorithm for finding suitable sets in graphs that are acyclic.  Given
our results so far we will maintain polynomial time guarantees simply
by considering only subsets $S,T$ of bounded size $|S|,|T|\leq m$.

\section{Edgewise Generic Identifiability} \label{sec:edgewise-id}

Aiming to strengthen the HTC and leverage the results of Section
\ref{sec:trek-sep}, the following theorem establishes a sufficient
condition for the generic identifiability of any set of incoming
edges to a fixed node. While in the process of preparing this manuscript we
discovered the work of \cite{chen2016}; our following theorem can be
seen as a generalization of their Theorem 1.

\begin{thm}\label{thm:edgewise-id}
  Let $G = (V, D, B)$ be a mixed graph, $v\in V$ be any vertex, and $\pa(v)=\{w_1,\dots,w_m\}$. Let $S\subset \pa(v)$ be such that for every $s\in S=\{s_1,\dots,s_\ell\}$ we have that $s\to v$ is generically (rationally) identifiable. Let $E\subset \pa(v)\setminus S$ and suppose there exists $Y\subset V\setminus(\{v\}\cup \sib(v))$ with $|Y|=|E|$ such that the following conditions hold,
  \begin{enumerate}[(i)]
  \item There exists a half-trek system from $Y$ to $E$ with no sided intersection, 
  \item $\tr(y) \cap \pa(v) \subset S\cup E$ for all $y\in Y$, and
  \item For every $y\in Y$, if $z\in \htr(v)\cap \pa(y)$ then $z\to y$ is generically (rationally) identifiable.
  \end{enumerate}
  Then for each $e\in E$ we have that $e\to v$ is generically (rationally) identifiable.
\end{thm}

\begin{proof}
  Let $(\Lambda, \Omega)$ be the matrices of indeterminants
  corresponding to $G$, and let
  $\Sigma = (I-\Lambda)^{-T}\Omega(I-\Lambda)^{-1}$ be the covariance
  matrix.  Recall our notation $\cT(v,w)$ for the set of treks from
  $v$ to $w$ in $G$.  By the trek rule (Prop. \ref{prop:trek-rule}),
  $\Sigma_{vw} = \sum_{\pi\in \cT(v,w)} \pi(\Lambda,\Omega)$ is the
  sum of monomials for treks from $v$ to $w$.

  Without loss of generality, let $E = \{w_1,\dots,w_k\}$ with $k\leq m$. Since $|E|=|Y|$ we enumerate $Y = \{y_1,\dots,y_k\}$. We denote the set of parents of $y_i$ that are half-trek reachable from $v$ as $H_i = \pa(y_i)\cap \htr(v) = \{h^i_1,\dots,h^i_{\ell_i}\}$.

  Our approach is to build a linear system of $k$ equations in $k$
  unknowns having a unique solution. Consider the set $\cT(y_1, v)$ of
  all treks between $y_1$ and $v$. Since
  $\tr(y_1) \cap \pa(v) \subset S\cup E$ and $y_1\bi v\not\in G$, all
  treks from $y_1$ to $v$ either end in a directed edge of the
  form $s_i \to v$ or $w_j\to v$, or must start in a directed edge of
  the form $y_1 \leftarrow h^1_i$. Now note that for any vertex
  $e\in E$,
  \begin{align*}
    \sum_{\pi\in \cT(y_1,e)} \pi(\Lambda,\Omega) - \sum_{i=1}^{\ell_1} \sum_{\pi\in \cT(h^1_i,e)} \pi(\Lambda,\Omega)\lambda_{h^1_iy_1} = \Sigma_{y_1e} - \sum_{i=1}^{\ell_1} \Sigma_{h^1_ie} \lambda_{h^1_iy_1}
  \end{align*}
  equals the sum of the monomials for all treks from $y_1$ to $e$ that
  do not start with a directed edge of the form
  $y_1 \leftarrow h^1_i$.  Arguing analogously when replacing $e$ with
  some $s\in S$, we find that the sum of all monomials for treks from
  $y_1$ to $v$ that do not start with an edge of the form
  $y_1\leftarrow h^1_i$ equals
  \begin{align*}
    \sum_{j=1}^\ell (\Sigma_{y_1s_j} - \sum_{i=1}^{\ell_1} \Sigma_{h^1_is_j} \lambda_{h^1_iy_1})\lambda_{s_jv} + \sum_{j=1}^k (\Sigma_{y_1w_j} - \sum_{i=1}^{\ell_1} \Sigma_{h^1_iw_j} \lambda_{h^1_iy_1})\lambda_{w_jv}.
  \end{align*}
  Now the sum over all treks between $y_1$ and $v$ that start with an
  edge of the form $y_1\leftarrow h^1_i$ is easily seen to be the
  quantity
  $\sum_{i=1}^{\ell_1}\Sigma_{vh^1_i}\lambda_{h^1_iy_1}$. Thus, 
  \begin{multline*}
    \Sigma_{y_1v} =
    \sum_{j=1}^\ell (\Sigma_{y_1s_j} - \sum_{i=1}^{\ell_1}
    \Sigma_{h^1_is_j} \lambda_{h^1_iy_1})\lambda_{s_jv} \\
    + \sum_{j=1}^k (\Sigma_{y_1w_j} - \sum_{i=1}^{\ell_1} \Sigma_{h^1_iw_j} \lambda_{h^1_iy_1})\lambda_{w_jv} + \sum_{i=1}^{\ell_1}\Sigma_{vh^1_i}\lambda_{h^1_iy_1}.
  \end{multline*}
  Rewriting this we have
  \begin{align*}
    \sum_{j=1}^k (\Sigma_{y_1w_j} - &\sum_{i=1}^{\ell_1} \Sigma_{h^1_iw_j} \lambda_{h^1_iy_1})\lambda_{w_jv} \\
                                    &= \Sigma_{y_1v} - \sum_{j=1}^\ell (\Sigma_{y_1s_j} - \sum_{i=1}^{\ell_1} \Sigma_{h^1_is_j} \lambda_{h^1_iy_1})\lambda_{s_jv} - \sum_{i=1}^{\ell_1}\Sigma_{vh^1_i}\lambda_{h^1_iy_1}.
  \end{align*}
  In the above equation, the only unknown parameters (that is, those
  not assumed to be generically identifiable), are the
  $\lambda_{w_jv}$. Hence we have exhibited one linear equation in the
  $k$ unknown parameters $\lambda_{w_jv}$. 

  Repeating the above argument for each of the $y_i$, we obtain $k$ linear
  equations in $k$ unknowns.  It remains to show that the system of
  equations is generically non-singular.  This amounts to showing
  generic invertibility for 
  the $k\times k$ matrix $A$ with entries
  \begin{align*}
    A_{ij} = \Sigma_{y_iw_j} - \sum_{k=1}^{\ell_i} \Sigma_{h^i_kw_j}\lambda_{h^i_k y_i}.
  \end{align*}
  The invertibility of $A$ follows from the existence of the
  half-trek system from $Y$ to $E$ with no sided intersection and
  Lemma \ref{lem:A-invertible} below.  We conclude that each $w_i\to v$ is
  generically (rationally) identifiable as claimed.
\end{proof}

The following lemma generalizes Lemma 2 from \cite{halftrek} and completes the proof of Theorem \ref{thm:edgewise-id}.

\begin{lem}\label{lem:A-invertible}
  Let $G = (V, D, B)$ be a mixed graph on $n$ nodes with associated covariance matrix $\Sigma$. Moreover, let $S=\{s_1,\dots,s_k\},\ T=\{t_1,\dots,t_k\}\subset V$. For every $1\leq i\leq k$ let $H_i=\{h^i_1,\dots,h^i_{\ell_i}\}\subset \pa(s_i)$. Suppose there exists a half-trek system from $S$ to $T$ with no sided intersection. Then the $k\times k$ matrix $A$ defined by
  \begin{align*}
    A_{ij} = \Sigma_{s_it_j} - \sum_{k=1}^{\ell_i}\Sigma_{h^i_kt_j}\lambda_{h^i_ks_i}
  \end{align*}
  is generically invertible.
\end{lem}

The proof of this lemma is deferred to Appendix \ref{app:proof-A-invertible}. Note that if $S = \emptyset$ and we require $E = \pa(v)$, then Theorem \ref{thm:edgewise-id} reduces to Theorem \ref{thm:htc-id}, the usual half-trek identifiability theorem (actually Theorem \ref{thm:edgewise-id} would still be slightly stronger as it does not require all incoming edges to a node in $Y\cap \htr(v)$ to be identified).

The conditions of Theorem \ref{thm:edgewise-id} can be easily checked in polynomial time using max-flow computations, just as with the standard half-trek criterion. Unfortunately, in general, we do not know for which subset $E\subset \pa(v)\setminus S$ we should be checking the conditions of Theorem \ref{thm:edgewise-id}. This, in practice, means that we will have to check all subsets $E\subset \pa(v)\setminus S$. There are, of course, exponentially many such subsets in general. If we are in a setting where we may assume that all vertices have bounded in-degree, then checking all subsets requires only polynomial time. In the case that in-degrees are not bounded, we may also maintain polynomial time complexity by only considering subsets $E$ of sufficiently large or small size. We provide pseudocode for an algorithm to iteratively identify the coefficients of a mixed graph leveraging Theorem \ref{thm:edgewise-id} in Algorithm \ref{alg:edgewise-id}.

\begin{algorithm}[t]
  \caption{Edgewise identification algorithm.}
  \label{alg:edgewise-id}
  \begin{algorithmic}[1]
    \STATE\textbf{Input:} A mixed graph $G= (V,D,B)$ with $V = \{1,\dots,n\}$ and a set of edges, $solvedEdges$, known to be generically identifiable.
    \REPEAT
    \FOR{$v \gets 1,\dots,n$}
    \STATE $unsolved \gets \{w\in V: w\to v\in G $ and $w\to v\not\in solvedEdges\}$.
    \STATE $maybeAllowed\gets$ \\
    \hspace{4mm}$\{y\in V\setminus(\{v\}\cup \sib(v)) :  z\in \htr(v)\cap \pa(y) \implies z\to y\in solvedEdges\}$
    \FOR{$E\subset unsolved$}
    \STATE $allowed \gets \{y\in maybeAllowed: \tr(y) \cap unsolved \subset E\}$
    \STATE $exists\gets $ Using max-flow computations, does there exist a half-trek system from $allowed$ to $E$ of size $|E|$ with no sided intersection?
    \IF{$exists$ is true}
    \STATE $solvedEdges\gets solvedEdges\cup\{e\to v: e\in E\}$
    \STATE Break out of the current loop
    \ENDIF
    \ENDFOR
    \ENDFOR
    \UNTIL{No additional edges have been added to $solvedEdges$ on the most recent loop.}
    \STATE\textbf{Output:} $solvedEdges$, the set of edges found to be generically (rationally) identifiable. 
  \end{algorithmic}
\end{algorithm}

\section{Edgewise Generic Nonidentifiability} \label{sec:edgewise-non-iden}

In prior sections we have focused solely on sufficient conditions for demonstrating the generic identifiability of edges in a mixed graph. This, of course, begs the question of if there are any complementary necessary conditions. That is, if there exist conditions that, when failed, show that a given edge is generically many-to-one. To our knowledge, the following is the only known necessary condition for generic identifiability and considers all parameters of a mixed graph $G$ simultaneously.

\begin{thm}[Theorem 2 of \citet{halftrek}] \label{thm:htc-nonidentifiability}
  Suppose $G=(V,D,B)$ is a mixed graph in which every family $(Y_v:v\in V)$ of subsets of the vertex set $V$ either contains a set $Y_v$ that fails to satisfy the half-trek criterion with respect to $v$ or contains a pair of sets $(Y_v,Y_w)$ with $v\in Y_w$ and $w\in Y_v$. Then the parameterization $\phi_G$ is generically infinite-to-one.
\end{thm}

This theorem operates by showing that, given its conditions, the Jacobian of the map $\phi_G$ fails to have full column rank and thus must have infinite-to-one fibers. Unfortunately this theorem does not give any indication regarding which edges are, in particular, generically infinite-to-one. The theorem below gives a simple condition which guarantees that a directed edge is generically infinite-to-one.

\begin{thm}\label{thm:edge-inf-to-one}
  Let $G=(V,D,B)$ be a mixed graph and let $v\to w\in G$. Suppose that
  for every $z\in V\setminus\{w\}$ we have either $z \bi w \in G$
  or $v$ is not half-trek reachable from $z$.  Let
  $\proj_{v\to w}$ be the projection
  $(\Lambda,\Omega)\mapsto \lambda_{vw}$ for $v\to w\in G$.  Then
  $\proj_{v\to w}(\cF(\Lambda,\Omega))$ is infinite for all
  $(\Lambda,\Omega)\in \Theta=\bR^D_{\mathrm{reg}}\times \mathit{PD}(B)$.
\end{thm}

\begin{proof}
  Let $(\Lambda,\Omega)\in\Theta$ and
  $\Sigma = \phi_G(\Lambda,\Omega) =
  (I-\Lambda)^{-T}\Omega(I-\Lambda)^{-1}$. We will show that for each
  matrix $\Gamma=(\gamma_{xy})\in \bR^D_{\mathrm{reg}}$ that agrees
  with $\Lambda$ in all but (possibly) the $(v,w)$ entry, we can find
  $\Psi\in \mathit{PD}(B)$ for which $\phi_G(\Gamma, \Psi) = \Sigma$.
  The claim then follows by noting that the choices for
  $\Gamma$ allow for infinitely many values of $\gamma_{vw}$.

  Let $\Gamma\in \bR^D_{\mathrm{reg}}$ be as above, and let
  $x\not=y\in V$ be such that $x\bi y \not\in G$.  Then
  \begin{align*}
    ((I-\Gamma)^T&\Sigma(I-\Gamma))_{xy} \\
    &= \sigma_{xy} - \sum_{z\in \pa(x)}\sigma_{yz}\gamma_{zx} - \sum_{z\in \pa(y)}\sigma_{xz}\gamma_{zy} + \sum_{z\in \pa(x)}\sum_{z'\in \pa(y)}\gamma_{zx}\gamma_{z'y}\sigma_{zz'}.
  \end{align*}
  Whenever $x,y\not=w$ then $\gamma_{zx}=\lambda_{zx}$ and $\gamma_{zy}=\lambda_{zy}$ in the above equation. Thus
  \begin{align*}
    0 = \Omega_{xy} = ((I-\Lambda)^T\Sigma(I-\Lambda))_{xy} = ((I-\Gamma)^T\Sigma(I-\Gamma))_{xy}.
  \end{align*}
  Next suppose, without loss of generality, that $x=w$ and $y\not=w$. Then, since $y$ is a non-sibling of $w$, we must have that $v$ is not half-trek reachable from $y$, and hence $\sigma_{vy} = \sum_{z\in \pa(y)}\sigma_{vz}\lambda_{zy}$.
  But then 
  \begin{align*}
    &((I-\Gamma)^T\Sigma(I-\Gamma))_{wy} \\
    &= \sigma_{wy} - \sum_{z\in \pa(w)}\sigma_{yz}\gamma_{zw} - \sum_{z\in \pa(y)}\sigma_{wz}\gamma_{zy} + \sum_{z\in \pa(w)}\sum_{z'\in \pa(y)}\gamma_{zw}\gamma_{z'y}\sigma_{zz'} \\
    &= -\sigma_{vy}\gamma_{vw} + \sum_{z'\in \pa(y)}\gamma_{vw}\lambda_{z'y}\sigma_{vz'} \\
    &\quad +\sigma_{wy} - \sum_{v\not=z\in \pa(w)}\sigma_{yz}\lambda_{zw} - \sum_{z\in \pa(y)}\sigma_{wz}\lambda_{zy} + \sum_{v\not=z\in \pa(w)}\sum_{z'\in \pa(y)}\lambda_{zw}\lambda_{z'y}\sigma_{zz'} \\
    &= -\gamma_{vw} (\sigma_{vy} - \sum_{z'\in \pa(y)}\lambda_{z'y}\sigma_{vz'}) \\
    &\quad +\sigma_{wy} - \sum_{v\not=z\in \pa(w)}\sigma_{yz}\lambda_{zw} - \sum_{z\in \pa(y)}\sigma_{wz}\lambda_{zy} + \sum_{v\not=z\in \pa(w)}\sum_{z'\in \pa(y)}\lambda_{zw}\lambda_{z'y}\sigma_{zz'}.
  \end{align*}
  Now since $\sigma_{vy} - \sum_{z'\in
    \pa(y)}\lambda_{z'y}\sigma_{vz'}=0$, we have that $0=-\gamma_{vw}
  (\sigma_{vy} - \sum_{z'\in \pa(y)}\lambda_{z'y}\sigma_{vz'}) =
  -\lambda_{vw} (\sigma_{vy} - \sum_{z'\in
    \pa(y)}\gamma_{z'y}\sigma_{vz'})$.   Therefore,
  \begin{align*}
    &((I-\Gamma)^T\Sigma(I-\Gamma))_{wy} \\
    &= -\lambda_{vw} (\sigma_{vy} - \sum_{z'\in \pa(y)}\lambda_{z'y}\sigma_{vz'}) \\
    &\quad +\sigma_{wy} - \sum_{v\not=z\in \pa(w)}\sigma_{yz}\lambda_{zw} - \sum_{z\in \pa(y)}\sigma_{wz}\lambda_{zy} + \sum_{v\not=z\in \pa(w)}\sum_{z'\in \pa(y)}\lambda_{zw}\lambda_{z'y}\sigma_{zz'} \\
    &=((I-\Lambda)^T\Sigma(I-\Lambda))_{wy} \\
    &=\Omega_{wy}=0.
  \end{align*}
  Let $\Psi = (I-\Gamma)^T\Sigma(I-\Gamma)$. We have just shown that
  $\Psi_{xy} = 0$ for every $x,y\in V$ such that $x\bi y\not\in G$. To
  see that $\Psi\in \mathit{PD}(B)$ it remains to show that $\Psi$ is
  positive definite. But this is obvious from its definition since
  $\Sigma$ is positive definite and $I-\Gamma$ is invertible.   We
  conclude that $\phi_G(\Gamma,\Psi) = \Sigma$ which proves the claim.
\end{proof}

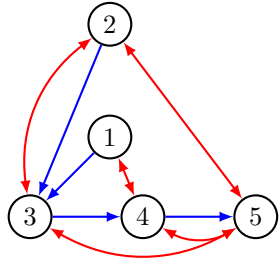
\begin{figure}[t]
  \centering
  \tikzset{
    every node/.style={circle, inner sep=1mm, minimum size=0.55cm, draw, thick, black, fill=white, text=black},
    every path/.style={thick}
  }
  \begin{tikzpicture}[align=center,node distance=1.5cm]
    \node [] (2) {2};
    \node [] (1) [below of=2] {1};
    \node [] (3) [below left of=1]    {3};
    \node [] (4) [right of=3]    {4};
    \node [] (5) [right of=4]    {5};

    \draw[blue] [-latex] (1) edge (3);
    \draw[blue] [-latex] (2) edge (3);
    \draw[blue] [-latex] (3) edge (4);
    \draw[blue] [-latex] (4) edge (5);
    
    \draw[red] [latex-latex] (1) edge (4);
    \draw[red] [latex-latex, bend right] (2) edge (3);
    \draw[red] [latex-latex] (2) edge (5);
    \draw[red] [latex-latex, bend right] (3) edge (5);
    \draw[red] [latex-latex, bend right] (4) edge (5);
  \end{tikzpicture}
  \caption{A graph in which all directed edges are identifiable except $2\to 3$. The $2\to 3$ edge can be shown to be infinite-to-one using Theorem \ref{thm:edge-inf-to-one}} \label{fig:infinite-to-one-edge}
\end{figure}

\begin{exmp}
  Let $G$ be the graph in Figure \ref{fig:infinite-to-one-edge}. Using the necessary condition of the HTC, Theorem \ref{thm:htc-nonidentifiability}, we find that $\phi_G$ is generically infinite-to-one. To identify which edges of $G$ are themselves infinite-to-one we use Theorem \ref{thm:edge-inf-to-one}. Doing so, one easily finds that the $2\to 3$ edge of $G$ is generically infinite-to-one. Indeed, using the edgewise identification techniques of Section \ref{sec:edgewise-id}, we see that all other directed edges of $G$ are generically identifiable so we have completely characterized which directed edges of $G$ are, and are not, generically identifiable.

  We stress, however, that Theorem \ref{thm:edge-inf-to-one} does not imply Theorem \ref{thm:htc-nonidentifiability}; that is, there are graphs $G$ for which Theorem \ref{thm:htc-nonidentifiability} shows $\phi_G$ is infinite-to-one but Theorem \ref{thm:edge-inf-to-one} cannot verify that any edges of $G$ are infinite-to-one.
\end{exmp}

\section{Computational Experiments} \label{sec:comp}

In this section we will provide some computational experiments that demonstrate the usefulness of our theorems in extending the applicability of the Half-Trek Criterion. All experiments below are carried out in the R programming language and use the R package SEMID available on CRAN, the Comprehensive R Archive Network \citep{R,semid}; our implementations of the TSID, EID, and EID$+$TSID algorithms described below are available on GitHub\footnote{Our implementations can be found at \url{https://github.com/Lucaweihs/SEMID}.} and will appear in the next release of the SEMID package on CRAN. We will be considering four different identification algorithms for checking generic identifiability:
\begin{enumerate}[(i)]
\item The standard half-trek criterion (HTC) algorithm.
\item The edgewise identification (EID) algorithm, displayed in Algorithm \ref{alg:edgewise-id}, where the input set of $solvedEdges$ is empty.
\item The trek-separation identification (TSID) algorithm. Similarly as for Algorithm
  \ref{alg:edgewise-id} this algorithm iteratively applies Theorem
  \ref{thm:t-sep-id-complex} until it fails to identify any additional
  edges.  (Since we are considering a small number of nodes there is
  no need to limit the size of sets $S$ and $T$ we are searching for
  in our computation.)
\item The EID$+$TSID algorithm. This algorithm alternates between the EID and TSID algorithms until it fails to identify any additional edges.
\end{enumerate}

We emphasize that when all of the directed edges, i.e., the matrix
$\Lambda$ is generically (rationally)  identifiable then we also have that
$\Omega=(I-\Lambda)^T\Sigma(I-\Lambda)$ is generically (rationally)
identifiable.  

In Table 1 of \cite{halftrek-supp}, the authors list all 112 acyclic non-isomorphic mixed graphs on 5 nodes which are generically identifiable but for which the half-trek criterion remains inconclusive even when using decomposition techniques. We run the EID, TSID, and EID$+$TSID algorithms upon the 112 inconclusive graphs and find that 23 can be declared generically identifiable by the EID algorithm, 0 by the TSID algorithm, and 98 by the EID$+$TSID algorithm. Thus it is only by using both the determinantal equations discovered by t-separation and the edgewise identification techniques that one sees a substantial increase in the number of graphs that can be declared generically identifiable.

We observe a similar trend to the above when allowing cyclic mixed
graphs. In Table 2 of \cite{halftrek-supp}, the authors list 75
randomly chosen, cyclic (i.e., containing a loop in the directed
part), mixed graphs that are known to be rationally identifiable but
cannot be certified so by the half-trek criterion. Of these 75 graphs,
4 are certified to be generically identifiable by the EID
algorithm, 0 by the TSID algorithm, and 34 by the EID$+$TSID
algorithm.

A listing of the 14 acyclic and 41 cyclic mixed graphs that could not be identified by the EID$+$TSID algorithm are listed as integer pairs $(d,b)\in\bN^2$ in Table \ref{table:eid-ts-inconclusive}. The algorithm to convert a pair $(d,b)$ in that table to a mixed graph $G$ on $n$ nodes is
\begin{enumerate}[1.]
\item For $v \gets 1,\dots,n$, for $w\gets 1,\dots,v-1,v+1,\dots,n$, do \\
  \phantom{\hspace{5mm}}Add edge $v\to w$ to $G$ if $d \text{ mod } 2=1$ \\
  \phantom{\hspace{5mm}}Replace $d$ with $\lfloor d/2 \rfloor$
\item For $v \gets 1,\dots,n-1$, for $w\gets v+1,\dots,n$, do \\
  \phantom{\hspace{5mm}}Add edge $v\bi w$ to $G$ if $b \text{ mod } 2=1$ \\
  \phantom{\hspace{5mm}}Replace $b$ with $\lfloor b/2 \rfloor$
\end{enumerate}
See Figure \ref{fig:eid-ts-u-example} for an example of a cyclic and acyclic graph that the EID$+$TSID algorithm fails to correctly certify as generically identifiable.

\begin{table}[t]
  \centering
  \begin{tabular}{c|cccc} \hline
    Acyclic & \multicolumn{3}{c}{Cyclic} \\
    \hline
    (4456, 113) & (345, 440) & (6629, 512) & (75321, 516) \\ 
    (360, 117) & (71329, 18) & (74536, 788) & (75398, 20) \\ 
    (6275, 172) & (81089, 0) & (5545, 96) & (70803, 896) \\ 
    (6307, 172) & (4714, 41) & (75112, 72) & (4457, 592) \\ 
    (6275, 188) & (70881, 80) & (74970, 4) & (74883, 522) \\ 
    (360, 369) & (74963, 512) & (4579, 384) & (350, 112) \\ 
    (4696, 401) & (74886, 268) & (70594, 65) & (74883, 2) \\ 
    (4936, 401) & (5058, 304) & (74921, 66) & (74950, 260) \\ 
    (4936, 402) & (70821, 513) & (70474, 640) & (74890, 38) \\ 
    (4680, 403) & (74915, 6) & (74922, 66) & (81076, 0) \\ 
    (840, 466) & (5267, 82) & (13160, 65) & (70851, 32) \\ 
    (5257, 658) & (76852, 128) & (4938, 448) & (1430, 120) \\ 
    (5257, 659) & (71075, 516) & (4730, 640) & (5251, 418) \\ 
    (4680, 914) & (4397, 897) & (70358, 1) &   \\
    \hline
  \end{tabular}

  \caption{Of the 112 acyclic and 75 cyclic mixed graphs on 5 nodes described in Tables 1 and 2 of \cite{halftrek-supp}, we display the 12 acyclic and 41 cyclic graphs which are known to be generically identifiable but for which the EID$+$TSID algorithm could not certify that all edges were generically identifiable. Each graph is encoded as a pair $(d,b)$, see text for details.} \label{table:eid-ts-inconclusive}
\end{table}

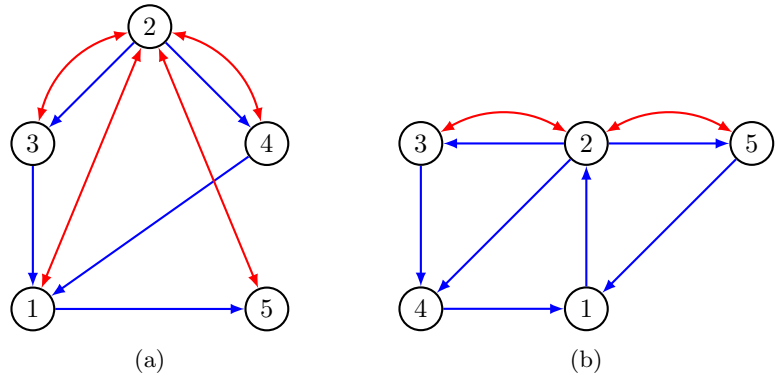
\begin{figure}[t]
  \begin{subfigure}[b]{0.4\textwidth}
    \centering
    \tikzset{
      every node/.style={circle, inner sep=1mm, minimum size=0.55cm, draw, thick, black, fill=white, text=black},
      every path/.style={thick}
    }
    \begin{tikzpicture}[align=center,node distance=2.2cm]
      \node [] (2) {2};
      \node [] (3) [below left of=2] {3};
      \node [] (4) [below right of=2]    {4};
      \node [] (1) [below of=3]    {1};
      \node [] (5) [below of=4]    {5};

      \draw[blue] [-latex] (2) edge (3);
      \draw[blue] [-latex] (2) edge (4);
      \draw[blue] [-latex] (3) edge (1);
      \draw[blue] [-latex] (4) edge (1);
      \draw[blue] [-latex] (1) edge (5);
      \draw[red] [latex-latex] (2) edge (1);
      \draw[red] [latex-latex, bend right] (2) edge (3);
      \draw[red] [latex-latex, bend left] (2) edge (4);
      \draw[red] [latex-latex] (2) edge (5);
    \end{tikzpicture}
    \caption{} \label{fig:eid-ts-u-acyc}
  \end{subfigure}
  ~ \hspace{0.04\textwidth}
  \begin{subfigure}[b]{0.4\textwidth}
    \centering
    \tikzset{
      every node/.style={circle, inner sep=1mm, minimum size=0.55cm, draw, thick, black, fill=white, text=black},
      every path/.style={thick}
    }
    \begin{tikzpicture}[align=center,node distance=2.2cm]
      \node [] (3) {3};
      \node [] (2) [right of=3] {2};
      \node [] (5) [right of=2]    {5};
      \node [] (4) [below of=3]    {4};
      \node [] (1) [below of=2]    {1};

      \draw[blue] [-latex] (1) edge (2);
      \draw[blue] [-latex] (2) edge (3);
      \draw[blue] [-latex] (2) edge (4);
      \draw[blue] [-latex] (2) edge (5);
      \draw[blue] [-latex] (3) edge (4);
      \draw[blue] [-latex] (4) edge (1);
      \draw[blue] [-latex] (5) edge (1);
      
      \draw[red] [latex-latex, bend right] (2) edge (3);
      \draw[red] [latex-latex, bend left] (2) edge (5);
    \end{tikzpicture}
    \caption{} \label{fig:eid-ts-u-cyc}
  \end{subfigure}
  
  \caption{Two graphs for which the EID$+$TSID algorithm is inconclusive. (a) is acyclic while (b) contains a cycle.}\label{fig:eid-ts-u-example}
\end{figure}

\section{Conclusion} \label{sec:conclusion}

By exploiting the trek-separation characterization of the vanishing of
subdeterminants of the covariance matrix $\Sigma$ corresponding to a
mixed graph $G$, we have shown that individual edge coefficients can
be generically identified by quotients of subdeterminants.  This
constitutes a generalization of instrumental variable techniques that
are derived from conditional independence.  We have also shown how
this information, in concert with a generalized half-trek criterion,
allows us to prove that substantially more graphs have all or some
subset of their parameters generically identifiable.

Our work on identification by ratios of determinants focuses on a
single edge coefficient.  However, it seems possible to give a
generalization that is in the spirit of the generalized instrumental
sets of \cite{Brito02}; see also \cite{vanderzander:2016}.  These
leverage several conditional independencies to find a linear equation
system that can be used to identify several edge coefficients
simultaneously, under specific assumptions on the interplay of the
conditional independencies and the edges to be identified.  We
illustrate the idea of how to do this using general determinants in
the following example.  However, a full exploration of this idea is
beyond the scope of this paper.  In particular, we are still lacking
mathematical tools that, in suitable generality, could be used to
certify that constructed linear equation systems have a unique
solution.

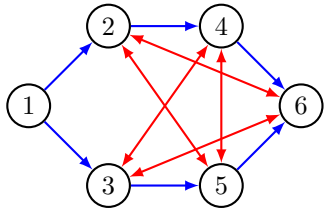
\begin{figure}[t]
  \centering
  \tikzset{
    every node/.style={circle, inner sep=1mm, minimum size=0.55cm, draw, thick, black, fill=white, text=black},
    every path/.style={thick}
  }
  \begin{tikzpicture}[align=center,node distance=1.5cm]
    \node [] (1) {1};
    \node [] (2) [above right of=1] {2};
    \node [] (3) [below right of=1]    {3};
    \node [] (4) [right of=2]    {4};
    \node [] (5) [right of=3]    {5};
    \node [] (6) [below right of=4]    {6};

    \draw[blue] [-latex] (1) edge (2);
    \draw[blue] [-latex] (1) edge (3);
    \draw[blue] [-latex] (2) edge (4);
    \draw[blue] [-latex] (3) edge (5);
    \draw[blue] [-latex] (4) edge (6);
    \draw[blue] [-latex] (5) edge (6);
    
    \draw[red] [latex-latex] (2) edge (6);
    \draw[red] [latex-latex] (3) edge (6);
    \draw[red] [latex-latex] (4) edge (5);
    \draw[red] [latex-latex] (2) edge (5);
    \draw[red] [latex-latex] (3) edge (4);
  \end{tikzpicture}
  \caption{A graph where the edges $4\to 6$ and $5\to 6$ can be simultaneously proven to be generically identifiable by solving a $2\times 2$ linear system of determinantal equations.} \label{fig:ts-multiple-edges-at-once}
\end{figure}

\begin{exmp}
  Let $G$ be the graph in Figure \ref{fig:ts-multiple-edges-at-once} with corresponding covariance matrix $\Sigma = (I-\Lambda)^{-T}\Omega(I-\Lambda)^{-1}$. Then, by similar considerations to those in Example \ref{exmp:t-sep-1}, one may show that
  \begin{align*}
    \left(\begin{matrix}
        |\Sigma_{\{3,5\},\{1,4\}}| & |\Sigma_{\{3,5\},\{1,5\}}| \\
        |\Sigma_{\{2,4\},\{1,4\}}| & |\Sigma_{\{2,4\},\{1,5\}}| \\
      \end{matrix}\right)
     \left(\begin{matrix}
         \lambda_{46} \\
         \lambda_{56}
       \end{matrix}\right) =
    \left(\begin{matrix}
         |\Sigma_{\{3,5\},\{1,6\}}| \\
         |\Sigma_{\{2,4\},\{1,6\}}|
       \end{matrix}\right).
  \end{align*}
  Using computer algebra we find that the $2\times 2$ matrix on the left hand side of the above equation has all non-zero polynomial entries, so that this is not equivalent to simply applying Theorem \ref{thm:t-sep-id-complex} for $4\to 6$ and $5\to 6$ separately, and has non-zero determinant. It follows that the above system is generically invertible and thus $\lambda_{46}$ and $\lambda_{56}$ are generically identifiable.
\end{exmp}

\section*{Acknowledgments}

This material is based on work started in June 2016 at the Mathematics
Research Communities (Week on Algebraic Statistics).  The work was
supported by the National Science Foundation under Grant Number DMS
1321794.

\appendix

\section{Proof of Lemma \ref{lem:0-det-generalization}} \label{app:proof-0-det-generalization}

We will require a known generalization of the Gessel-Viennot-Lindstr\"{o}m lemma which we now state.

\begin{defn}
  Let $G=(V,D)$ be a directed graph with vertices $V=\{1,\dots,n\}$ and corresponding matrix of indeterminants $\Lambda$. Let $\pi = v_1\to v_2\to\dots\to v_\ell$ be a directed path in $G$. Then define the \emph{loop erased} path $LE(\pi)$ corresponding to $\pi$ recursively as follows. If $\pi$ contains no loops then $\pi = LE(\pi)$. Otherwise there exist indices $1\leq i < j \leq \ell$ such that $v_i=v_j$. Then $LE(\pi) = LE(\pi')$ where $\pi' = v_1\to v_2 \to\dots\to v_i \to v_{j+1} \to \dots \to v_\ell$. It can be shown that $LE(\pi)$ is well defined (i.e. is independent of the ordering of the above recursion).
\end{defn}

\begin{lem}[Gessel-Viennot-Lindstr\"{o}m Generalization, Theorem 6.1 of \cite{fomin2001}]
  Let $G=(V,D)$ be a directed graph with vertices $V=\{1,\dots,n\}$ and corresponding matrix of indeterminants $\Lambda$. Define $\Psi = (I-\Lambda)^{-1}$ and for any directed path $\pi$ in $G$ define the path polynomial $\pi(\Lambda)  = \prod_{w\to v\in \pi}\lambda_{wv}$. Then for any $S=\{s_1,\dots,s_k\},T=\{t_1,\dots,t_k\}\subset V$ we have that
  \begin{align*}
    |\Psi_{S,T}| = \sum_{\tau \in P_n}\sign(\tau) \sum_{\substack{s_1\overset{\pi_1}{\longrightarrow} t_{\tau(1)},\dots,s_k\overset{\pi_k}{\longrightarrow} t_{\tau(k)} \\ i<j \implies \pi_j \cap LE(\pi_i) = \emptyset}} \pi_1(\Lambda)\hdots\pi_k(\Lambda),
  \end{align*}
  here the above inner sum is over all directed path systems $\Pi=\{\pi_1,\dots,\pi_k\}$ with $\pi_i$ going from $s_i$ to $t_{\tau(i)}$ for all $i$, where $\pi_j$ and $LE(\pi_i)$ share no vertices for $i<j$. Hence $|\Psi_{S,T}| = 0$ if and only if every system of directed paths from $S$ to $T$ has two paths which share a vertex.
\end{lem}

The remaining proof of Lemma \ref{lem:0-det-generalization} proceeds in several parts and closely follows similar results in \cite{sullivant2010} and \cite{draisma2013}. As such we will state several lemmas whose proofs require only small modifications of existing results (such as replacing the standard Gessel-Viennot-Lindstr\"{o}m Lemma with its generalization above). In such cases we will simply direct the reader to the corresponding proof and sketch the necessary modifications.

\begin{defn}
  Let $G=(V,D,B)$ be a mixed graph and let $U\subset D$. We say a trek $\pi$ in $G$ \emph{avoids $U$ on the left (right)} if the left (right) side of $\pi$ uses no edges from $U$. Similarly we say a system of treks $\Pi$ in $G$ \emph{avoids $U$ on the left (right)} if every trek $\pi\in \Pi$ avoids $U$ on the left (right). If $U_L,U_R\subset D$ we say that a \emph{trek (or trek system) avoids $(U_L,U_R)$} if it avoids $U_L$ on the left and $U_R$ on the right.
\end{defn}

\begin{lem}\label{lem:empty-bidirected-part}
  Let $G=(V,D,B)$ be a mixed graph and let $\Lambda,\Omega$ be $n\times n$ matrices of indeterminants corresponding to the directed and bidirected parts of $G$ respectively. Suppose that $B=\emptyset$ so that $\Omega$ is diagonal. Letting $D_L,D_R,\Lambda^L,\Lambda^R,\Gamma,$ and $G^*_{\mathrm{flow}}$ be as in Lemma \ref{lem:0-det-generalization} we have that for any $S,T\subset V$ with $|S|=|T|=k$, $|\Gamma_{S,T}|=0$ if and only if
  the max-flow from $S$ to $T'$ in $G^*_{\mathrm{flow}}$ is $<k$.
\end{lem}

\begin{proof}
  In the following, whenever we say ``As in x,'' we mean ``As in the proof of x in \citet{sullivant2010}.''
  
  As in Lemma 3.2, we have $|\Gamma_{S,T}|=0$ if and only if for every set $A\subset V$ with $|A|=K$ we have $|((I-\Lambda^L)^{-1})_{S,A}|=0$ or $|((I-\Lambda^R)^{-1})_{A,T}|=0$. As in Prop. 3.5, using the above result, and applying our version of the Gessel-Viennot-Lindstr\"{o}m Lemma, we have that $|\Gamma_{S,T}|=0$ if and only if every system of (simple) treks avoiding $(D\setminus D_L, D\setminus D_R)$ has sided intersection.

  Now noticing that $B=\emptyset$ simplifies the definition of $G^*_{\mathrm{flow}}$, we have as in Prop. 3.5 that the (simple) treks from $u$ to $v$ avoiding $(D\setminus D_L,D\setminus D_{R})$ in $G$ are in bijective correspondence with directed paths from $u$ to $v'$ in $G^*_{\mathrm{flow}}$. Finally the result follows by noticing that max-flow systems from $S$ to $T'$ in $G^*_{\mathrm{flow}}$ of size $k$ correspond to systems of treks from $S$ to $T$ avoiding $(D\setminus D_L, D\setminus D_R)$ with no-sided intersection (that is, if one exists so does the other). Combining the above if and only if statements, the result then follows.
\end{proof}

We have now proven our desired result in the case $B=\emptyset$, it remains to show that this implies the case $B\not=\emptyset$. To this end, we say that $\tilde{G}=(\tilde{V}, \tilde{D},\tilde{B})$ is the \emph{bidirected subdivision} of $G = (V,D,B)$ if it equals $G$ but where we have replaced every bidirected edge $i\bi j\in G$ with a vertex $v_{(i,j)}$ and two edges $v_{(i,j)}\to i$ and $v_{(i,j)}\to j$ (with associated parameters $\tilde{\omega}_{(i,j),(i,j)},\tilde{\lambda}_{(i,j)i},\tilde{\lambda}_{(i,j)j})$. Note that we have subdivided every bidirected edge into two directed edges which motivates the naming convention. Let $\tilde{D}_L$ and $\tilde{D}_R$ be equal to $D_L$ and $D_R$ respectively but where we have also added in the new edges $v_{(i,j)}\to i$ and $v_{(i,j)}\to j$ for every $i\bi j\in G$. Let $\tilde{\Lambda},\tilde{\Omega}$ be matrices of indeterminants corresponding to $\tilde{G}$ and let $\tilde{\Lambda}^L$, $\tilde{\Lambda}^R$ correspond to $\tilde{D}_L,\tilde{D}_R$ just as for $G$. We now have the following result that relates $G$ and $\tilde{G}$.

\begin{lem}\label{lem:bidirected-sub}
  Let $G$, $\tilde{G}$ be as in the prior paragraph. Then letting $\tilde{\Gamma} = (I-\tilde{\Lambda}^L)^{-T}\tilde{\Omega}(I-\tilde{\Lambda}^R)^{-1}$ we have that, for any polynomial $f$ taking, as input, an $n\times n$ matrix of variables, we have that $f(\Gamma)=0$ if and only if $f(\tilde{\Gamma}) = 0$. In particular, since the subdeterminant of a matrix is a polynomial in the entries of the matrix, we have that for any $S,T\subset V$ with $|S|=|T|=k$, $|\Gamma_{S,T}|=0$ if and only if $|\tilde{\Gamma}_{S,T}|=0$.
\end{lem}

\begin{proof}
  This proof follows, essentially exactly, as the first part of the proof of Prop. 2.5 in \citet{draisma2013}.
\end{proof}

Now we show that the above subdivision trick produces a graph $\tilde{G}^*_{\mathrm{flow}}$ for which the max-flow between vertex sets is the same as for $G^*_{\mathrm{flow}}$.

\begin{lem} \label{lem:bisect-same-flows}
  Consider the graphs $G^*_{\mathrm{flow}}=(V^*,D^*)$ from the Lemma \ref{lem:0-det-generalization} statement and let $\tilde{G}^*_{\mathrm{flow}}=(\tilde{V}^*, \tilde{D}^*)$ be corresponding flow graph for the bidirected subdivision $\tilde{G}$ of $G$ . Let $S=\{s_1,\dots,s_k\},\ T=\{t_1,\dots,t_k\}\subset V$. Then the maximum flow from $S$ to $T'=\{t_1',\dots,t_k'\}$ in $G^*_{\mathrm{flow}}$ equals the maximum flow from $S$ to $T'$ in $\tilde{G}^*_{\mathrm{flow}}$.
\end{lem}

\begin{proof}
  Recall that a flow system on a graph is an assignment of flow to the edges and vertices of the graph satisfying the usual flow constraints. Also recall that, for graphs with integral capacities, there always exists a max-flow system between subsets of nodes for which all flow assignments upon edges and vertices take values in $\bN$. We will show that any (integral valued) max-flow system from $S$ to $T'$ in $\tilde{G}^*_{\mathrm{flow}}$ corresponds to a unique flow system in $G^*_{\mathrm{flow}}$ with the same total flow and vice-versa. Our result then follows. 

  Let $\tilde{\cF}$ be a max-flow system from $S$ to $T'$ on $\tilde{G}^*_{\mathrm{flow}}$ from $S$ to $T'$ with integral flow assignments. Since $\tilde{G}^*_{\mathrm{flow}}$ and $G^*_{\mathrm{flow}}$ have all capacities equal to 1 it follows that $\tilde{\cF}$ assigns either 0 or 1 flow to all edges and vertices in the graph.

  We now construct a flow system $\cF$ on $G^*_{\mathrm{flow}}$ with the same capacity. First let $\cF$ assign the same capacity to all edges and vertices that $\cF$ shares with $\tilde{\cF}$. Note that if $\tilde{\cF}$ does not assign any flow to any of the edges incoming to the vertices $v_{(i,j)}$ then $\cF$ already corresponds to a flow system on $G^*_{\mathrm{flow}}$ with the same total flow. Suppose otherwise that $\tilde{\cF}$ assigns 1 unit of flow to the edges $\{a_1 \to v_{a_1b_1'}, \dots, a_k\to v_{a_kb_k'}\}$. Since $v_{(i,j)}$ and the $a_i$ have capacity 1 it follows that $a_i\not=a_j$ and $v_{a_ib_i'}\not=v_{a_ib_i'}$ for all $i\not=j$. For each edge $a_i \to v_{a_ib_i'}$, since $v_{a_ib_i'}$ has two outgoing edges $v_{a_ib_i'}\to a_i'$ and $v_{a_ib_i'}\to b_i'$, there are two possible cases:
  \begin{itemize}
  \item Case 1: $\tilde{\cF}$ assigns 1 flow to $v_{a_ib_i'}\to a_i'$. 
    
    In this case assign a flow of 1 to the edge $a_i\to a_i'$ in $\cF$.
  \item Case 2: $\tilde{\cF}$ assigns 1 flow to $v_{a_ib_i'}\to b_i'$.
    
    In this case assign a flow of 1 to the edge $a_i\to b_i'$ in $\cF$.
  \end{itemize}
  It is easy to check that $\cF$ is indeed a valid flow system on $G^*_{\mathrm{flow}}$ with the same flow as $\tilde{\cF}$.

  To see the oppose direction let $\cF$ be a max-flow system from $S$ to $T'$ on $G^*_{\mathrm{flow}}$ from $S$ to $T'$ with integral flow assignments. We now construct a flow system $\tilde{\cF}$ on $\tilde{G}^*_{\mathrm{flow}}$ with the same capacity. As before, first let $\tilde{\cF}$ assign the same capacity to all edges and vertices that $\tilde{\cF}$ shares with $\cF$. Note that if $\cF$ does not assign any flow to any of the edges $a\to b'$ for $(a,b)\in B$ then $\tilde{\cF}$ already corresponds to a flow system on $G^*_{\mathrm{flow}}$ with the same total flow. Suppose otherwise that $\tilde{\cF}$ assigns 1 unit of flow to the edges $E=\{a_1 \to b_1', \dots, a_k\to b_k'\}$ with $(a_i,b_i)\in B$ for all $i$. Since all vertices in $\cF$ have capacity 1 we must have that $a_i\not=a_j$ and $b_i\not=b_j$ for all $i\not=j$. There are two possible cases:
  \begin{itemize}
  \item Case 1: $a_i\to b_i'\in E$ and $b_i\to a_i \not\in E$. 
    
    In this case assign a flow of 1 along the path $a_i\to v_{a_ib_i} \to b_i'$ in $\tilde{\cF}$.
  \item Case 2: $a_i\to b_i'\in E$ and $b_i\to a_i\in E$. 
    
    In this case assign a flow of 1 to the edges $a_i\to a_i'$ and $b_i\to b_i'$ in $\tilde{\cF}$.
  \end{itemize}
  One may now check that $\tilde{\cF}$ is a valid flow system on $\tilde{G}^*_{\mathrm{flow}}$ with the same flow as $\cF$.
\end{proof}

Finally we are in a position to easily prove Lemma \ref{lem:0-det-generalization}. Note that, by Lemma \ref{lem:bidirected-sub} we have that $|\Gamma_{S,T}| = 0$ if and only if $|\tilde{\Gamma}_{S,T}| = 0$. By Lemma \ref{lem:empty-bidirected-part} we have that $|\tilde{\Gamma}_{S,T}| = 0$ if and only if the max-flow from $S$ to $T'$ in $\tilde{G}^*_{\mathrm{flow}}$ equals $|S|=k$. Finally Lemma \ref{lem:bisect-same-flows} gives us that the max-flow from $S$ to $T'$ in $\tilde{G}^*_{\mathrm{flow}}$ equals the max-flow from $S$ to $T'$ in $G^*_{\mathrm{flow}}$. Hence we have that $|\Gamma_{S,T}| = 0$ if and only if the max-flow from $S$ to $T'$ in $G^*_{\mathrm{flow}}$ equals $k$, this was our desired statement.

\section{Proof of Lemma \ref{lem:A-invertible}} \label{app:proof-A-invertible}

The proof of this lemma follows almost identically as the proof of Lemma 2 in \cite{halftrek}. We simply restate the arguments there in our setting. For any $v,w\in V$ let $\cH(v,w)$ be the set of half treks from $v$ to $w$ in $G$. Also let $\cT_{ij}$ be the set of all treks from $s_i$ to $t_j$ in $G$ which do not begin with an edge of the form $s_i \leftarrow h^i_k$ for any $1\leq k\leq \ell_i$. Then it is easy to see that $\cH(s_i,t_j) \subset \cT_{ij}$. Now, by the Trek Rule (Proposition \ref{prop:trek-rule}), we have that
\begin{align*}
  A_{ij} = \sum_{\pi\in \cT_{ij}} \pi(\Lambda,\Omega).
\end{align*}
Now for any system of treks $\Pi$ define the monomial
\begin{align*}
  \Pi(\Lambda,\Omega) = \prod_{\pi\in\Pi} \pi(\Lambda,\Omega).
\end{align*}
Then, by Leibniz's formula for the determinant, we have that
\begin{align}
  |A| = \sum_{\Pi}(-1)^{\sign(\Pi)}\Pi(\Lambda,\Omega) \label{eq:A-det-sum}
\end{align}
where the above sum is over all trek systems $\Pi$ from $S$ to $T$ using treks only in the set $\cup_{1\leq i,j\leq k}\cT_{ij}$; here the $\sign(\Pi)$ is the sign of the permutation that writes $t_1,\dots,t_k$ in the order of their appearance as targets of the treks in $\Pi$.

By assumption, there exists a half-trek system from $S$ to $T$ with no-sided intersection. Since such a system exists, let $\Pi$ be a half-trek system of minimum total length among all such half-trek systems. Since $\cH(s_i,t_j)\subset \cT_{ij}$ for all $i,j$ it follows that $\Pi$ is included as one of the trek systems in the summation \eqref{eq:A-det-sum}. Let $\Psi$ be any system of treks from $S$ to $T$ such that $\Psi(\Lambda,\Omega) = \Pi(\Lambda,\Omega)$. Lemma 1 of \cite{halftrek} proves that we must have $\Psi = \Pi$ so that $\Pi$ is the unique system of treks from $S$ to $T$ with corresponding trek monomial $\Pi(\Lambda,\Omega)$. It thus follows that the coefficient of the monomial $\Pi(\Lambda,\Omega)$ in $|A|$ is $(-1)^{\sign(\Pi)}$ and thus $|A|$ is not the zero polynomial (or power series if the sum is infinite). Hence, for generic choices of $(\Lambda,\Omega)$, we have that $|A|\not=0$ so that $A$ is generically invertible.

\bibliographystyle{abbrvnat}
\bibliography{half_trek}
\end{document}